\documentclass[12pt, reqno]{amsart}
\usepackage{amsfonts,amsmath,amsthm,oldgerm,amssymb,amscd}

\usepackage[all,cmtip]{xy}
\usepackage{fancyhdr}
\usepackage{psfrag}
\usepackage{graphicx}
\usepackage{amsbsy}
\usepackage{indentfirst}
\usepackage{float}
\usepackage{latexsym}
\usepackage{graphics}
\usepackage{verbatim}
\usepackage{fixmath}
\newcommand{\Q}{{\mathbb Q}}
\newcommand{\R}{{\mathbb R}}
\newcommand{\Z}{{\mathbb Z}}
\newcommand{\C}{{\mathbb C}}
\newcommand{\N}{{\mathbb N}}

\newcommand{\V}{{\mathbb V}}

\newcommand{\G}{{\mathbf G}}

\renewcommand{\O}{{\mathcal O}}

\newcommand{\cE}{{\mathcal E}}
\newcommand{\cB}{{\mathcal B}}

\newtheorem{theorem}{Theorem}%[section] (If you want theorem numbered
\newtheorem{lemma}{Lemma}%               with section number.  Same
\newtheorem{corollary}{Corollary}%       goes for lemmas, etc.)
\newtheorem{proposition}[theorem]{Proposition}
\theoremstyle{definition}

\newtheorem*{example}{Example}
\newtheorem{definition}{Definition}
\numberwithin{definition}{section}

\numberwithin{equation}{section}

\theoremstyle{remark}
\newtheorem*{remark}{Remark}
\newtheorem*{acknowledgement}{Acknowledgement}

\begin{document}

\title[Universal Torelli for elliptic surfaces]{A universal Torelli theorem
  \\for 
  elliptic surfaces}

\author{C.~S.~Rajan}

\address{Tata Institute of Fundamental  Research, Homi Bhabha Road,
Bombay - 400 005, INDIA.}  \email{rajan@math.tifr.res.in}

\author{S.~Subramanian}

\address{Tata Institute of Fundamental  Research, Homi Bhabha Road,
Bombay - 400 005, INDIA.}  \email{subramnn@math.tifr.res.in}

\begin{abstract} Given two semistable
elliptic surfaces over a curve $C$ defined over a field of
characteristic zero or finitely generated over its prime field,  we
show that any compatible family of effective isometries of the
N{\'e}ron-Severi lattices of the base changed elliptic surfaces for
all finite separable maps $B\to C$ arises from an isomorphism of the
elliptic surfaces. Without the effectivity hypothesis, we show that
the two elliptic surfaces are isomorphic. 

We also determine the group of universal automorphisms of a semistable
elliptic surface. In particular, this includes showing that the
Picard-Lefschetz transformations corresponding to an irreducible
component of a singular fibre, can be extended as universal
isometries. In the process, we get a family of  homomorphisms of the
affine Weyl group associated to $\tilde{A}_{n-1}$ to that of
$\tilde{A}_{dn-1}$, indexed by natural numbers $d$, which are closed
under composition.

\end{abstract}

\maketitle

\section{Introduction}
Let $X$ be a compact, connected, oriented K\"{a}hler manifold of dimension $d$
with an integral K\"ahler form $\omega\in H^2(X, \Z)\cap H^{1,1}(X,
{\Bbb C})$. 
 The intersection
product induces a graded algebra structure on the cohomology algebra
\[ H(X)=\oplus_{i=0}^{2d} H^i(X, {\Bbb Z}), \] where $H^i(X, {\Bbb
Z})$ are the singular cohomology groups of $X$.  Hodge theory provides
a filtration of the complex cohomology groups  $H^i(X,{\Bbb C})$.  The
K\"{a}hler form induces a polarization of the Hodge structure. The
classical Torelli question is whether the space $X$ can be recovered
from the polarized Hodge structure of the cohomology algebra equipped
with the intersection product. 

When $X$ is a compact, connected Riemann surface the Torelli question
has an affirmative answer: the Riemann surface is determined by its
associated polarized Hodge structure.

Now, let $\pi: ~X\to C$ be an elliptic surface over a smooth,
projective curve $C$ over $\C$.  Different aspects of  the  Torelli
problem  have been well studied for elliptic surfaces.  It is known,
for instance, that Torelli does not hold for elliptic surfaces (see
Example \ref{torellifailure}). 

One of the problems that arises  with the Torelli question for
elliptic surfaces is that the  N{\'e}ron-Severi group of the surface
$X$  is not sufficiently large enough to distinguish between the
surfaces. To rectify this problem, we argue  in analogy with Tate's
isogeny conjecture, or with Grothendieck's use of base changes. This
leads us  to base change the elliptic surfaces by finite maps $B\to
C$, so that the  N{\'e}ron-Severi group of the base changed elliptic
surfaces becomes larger. 

We can consider all base changes of the   elliptic surfaces by finite
separable morphisms of the base curve, and a family of compatible
isometries between the  N{\'e}ron-Severi groups of the base changed
surfaces.  For the existence of compatible isometries, we need to work
with  semistable elliptic surfaces.

The resulting object will carry an action of the absolute Galois group
of the generic point of $C$, and the isometries will have to be
equivariant with respect to the action of the Galois group.
Considering the whole family of N\'eron-Severi lattices carries the
risk that the collection of effective isometries of the
N{\'e}ron-Severi lattices can become larger corresponding to the
growth of the  N{\'e}ron-Severi group.   Miraculously, this does not
happen. Working over a field $k$, which is of characteristic zero or
finitely generated over its prime field,  we show that
compatible, effective  isometries of the  N{\'e}ron-Severi lattices of
the  base changes for all base changes of the base curve, arises from
an isomorphism of the elliptic surfaces. The introduction of the
effectivity hypothesis is critical, enlarging the scope of the
theorem, making it more natural and compelling, and follows the use of
the effectivity hypothesis  for the Torelli theorem for $K3$-surfaces
proved  by   Piatetskii-Shapiro and Shafarevich. 

We next consider describing  the group of universal isometries of a
semistable elliptic surface dropping the effectivity assumption on the
isometries. We observe that   the Picard-Lefschetz transformations
based on the irreducible components of the singular fibres can be
extended to give compatible, isometries of the   N{\'e}ron-Severi
lattices of the  base changes  for all base changes of the base
curve. The Picard-Lefschetz reflections of a singular fibre of type
$I_n$ generates the affine Weyl group of $\tilde{A}_{n-1}$. 
The universality of Picard-Lefschetz reflections defines a 
family of representations of the affine Weyl group
$\tilde{A}_{n-1}$ to $\tilde{A}_{d(n-1)}$ for any natural number $d$
which are closed under composition.  Such maps correspond to the base
change by a map of degree $d$ of the base curve, totally ramified of
degree $d$ at the point corresponding to the singular fibre. The study
of these representations allow us to determine the group of universal
automorphisms of the  N{\'e}ron-Severi lattices attached to base
changes of a semistable elliptic surface. 

\subsubsection{Outline of the paper.}  We now give an outline of the
paper. In section \ref{sec:theorems}, we introduce the notion of
universal N{\'e}ron-Severi groups and state the main theorems. In
section \ref{sec:fibral}, using base changes,  we first show that
universal isometries preserve fibres. We then show that univeral
isometry determines the Kodaira-N\'eron type of the singular
fibres. In section \ref{sec:isogeny},  appealing  to theorems proving
the Tate isogeny conjecture, we establish that the generic fibres are
isogenous. Using the extra information coming from our hypothesis, it
is shown  in section \ref{sec:utt},  that the elliptic surfaces are
indeed isomorphic.   

Using the effectivity hypothesis, one concludes that sections are
mapped to sections and the irreducible components of singular fibres
are preserved by the universal isometry (see section
\ref{sec:effectivity}). The fact that the elliptic surfaces are
isomorphic allows us to compose the universal isometry with
itself. Making use of the fact that torsion elements are determined by
their intersection with the components of singular fibres, allows us
to conclude in section \ref{sec:fibral2},  that the action of the
universal isometry on torsion and the fibral divisors is geometric. 
The final proof of the effective universal Torelli theorem carried out
in section \ref{sec:uett} rests on the use
of  a Galois theoretic argument together with the geometry and
arithmetic around the narrow Mordell-Weil group of the generic fibre. 

The second half of the paper studies the representation theoretic and
geometric aspects of the  Picard-Lefschetz reflections based on the
irreducible components of the singular fibres. We first determine the
base change map on fibral divisors in section \ref{sec:basechange}. 
This description allows us to arrive in an inductive manner the 
definition of the lifts of Picard-Lefschetz reflections to universal
isometries. However it is more convenient to represent these
reflections in terms of usual permutation notation, allowing us to
come up with an alternate definition of the universal
Picard-Lefschetz isometries. These facts and the reprsentation
theoretic aspects of the affine Weyl group of type $\tilde{A}$ that
arise are studied in section \ref{sec:repn}. 

In section \ref{sec:upl}, we show that the lifts of Picard-Lefschetz
reflections we have defined indeed define isometries of the universal
N\'eron-Severi group. In the last section \ref{sec:guti}, we
determine the group of isometries of the universal
N\'eron-Severi group of a semistable elliptic surface. 

\section{Elliptic surfaces and the main theorems}\label{sec:theorems}
For a variety $Z$ defined over $k$, let $\bar{Z}=Z\times_k \bar{k}$, 
where $\bar{k}$ is a fixed separable
closure of $k$. If $f:Z\to Y$ is a morphism of schemes defined over
$k$, let $\bar{f}: \bar{Z}\to \bar{Y}$ denote the base change of the
morphism $f$ to $\bar{k}$. The structure sheaf of a variety $X$ is
denoted by ${\mathcal O}_X$. 

Let $C$ be a connected,  smooth, projective curve over a field $k$. An elliptic
surface  $\cE: X\xrightarrow{\pi} C$ is a non-singular, projective surface $X$
defined over $k$ together with a surjective morphism 
$\pi: X\to C$ such that the following conditions are satisfied: 

\begin{itemize}\label{assumptions}
%\item $\pi_*({\mathcal O}_X)={\mathcal O}_C$. 

\item The generic fibre $E$ over the function field
$K=k(C)$ of $C$ is a smooth, irreducible curve  of
genus $1$. %Almost all the fibres of $\pi$ are nonsingular elliptic curves. 

\item The map $\pi:X\to C$ has a section. 

\item The curve $C$ has a $k$-rational point. 

\item   $\bar{\pi}$ is relatively minimal, i.e., there is no irreducible,
  rational curve $D$ on $\bar{X}$ 
with self-intersection $D^2=-1$ contained in a
  fibre of $\bar{\pi}$. 

\item  The $j$-invariant of the generic fibre is not algebraic over
  $k$. Equivalently the elliptic surface is not potentially
  iso-trivial. 

\end{itemize}

We refer to the excellent surveys (\cite{M, SS}) and (\cite[Chapter
III]{Si2}) for information about
elliptic surfaces (primarily 
over algebraically closed ground fields).  

\subsection{Base change}
Let $\cB_C$ be the collection of triples $(l, B, b)$ consisting of the
following
\begin{itemize}
\item $l$ is a finite separable extension of $k$ contained inside $\bar{k}$. 

\item $B$ is a geometrically integral, 
regular projective curve defined over $l$.

\item $b:B\to C\times_k {\rm Spec} (l)$ is a finite, separable
  morphism. 
\end{itemize}
When the situation is clear, we drop the use of the subscript $C$, and
also simply refer the morphism $b \in \cB$. 

For $b\in \cB_C$, 
let $X_b$ be the relatively minimal regular model in the birational
equivalence class of the base
change surface $X^b:=(X\times_k {\rm Spec}
(l))\times_{C\times_k {\rm Spec} (l)} B\to B$.  The elliptic surface
$\cE_b: X_b\to B$ can be considered as the unique relatively minimal regular
elliptic surface over $B$ with generic fibre the curve $E$
considered as an elliptic curve over $l(B)$. 

\subsection{Semistable elliptic surfaces}
For a place $t$ of $k(C)$, let ${\mathcal O}_{C,t}$ denote the local ring
at $t$. The elliptic surface defines an elliptic curve $E_t$ over
${\mathcal O}_{C,t}$. 
Define an elliptic surface $\cE$ to have {\em semistable} reduction at
$t$, if the elliptic curve $E_t$ has either good or split multiplicative
reduction modulo the maximal ideal in  ${\mathcal O}_{C,t}$. 

This amounts to saying that the fibre at $t$, is either an elliptic
curve, or is  of type $I_n, ~n\geq 1$ in the Kodaira-N\'{e}ron
classification of singular fibres. At a place having singular
reduction of type $I_n, n\geq 3$,  the special fibre $X_t$ of $X$ at
$t$, is a reduced cycle consisting of $n$ smooth, rational curves,
with self-intersection $-2$, and 
each curve intersecting its neighbours with multiplicity one.  

Define an  elliptic surface $\cE: X\to C$ to be {\em semistable}, if
it has semistable reduction at all places $v$ of $C$. 

{\em Notation.} The ramification locus of $\pi$ is usually denoted by
$S$. For $t\in S\subset C(k)$, the  Kodaira-N\'{e}ron type of the singular fibre
is denoted by $I_{n_t}$. The irreducible components of the fibre
at $t$ are denoted by $v^t_i, ~i\in \Z/n_t\Z$, where the component
$v^t_0$ is the component intersecting the zero section.  If $n_t\geq 3$, then
\[(v^t_i)^2=-2, \quad v^t_iv^t_{i+1}=1 \quad \mbox{for} \quad i\in \Z/n_t\Z.\]
At times the superscript $t$ is dropped. When base changes are
involved, $w$ is used instead of $v$ to denote 
the components of the base changed surface.

The following well known
theorems ensuring the existence of semistable base change
and properties of semistable surfaces under further base change are
crucial to the formulation and proof of the results of this paper: 

\begin{theorem}\label{ssbc}
(i) Given an elliptic
surface $X\to C$, there is a triple  $(l, B, b)\in \cB$,  such that
the base changed surface $\cE_b: X_b\to B$ is semistable. 

(ii) Suppose $X\to C$ is a semistable elliptic surface, and  $(l, B, b)\in
\cB$. Then $X_b$ is the minimal desingularization of   
$X^b:=(X\times_k {\rm Spec}(l))\times_{C\times_k {\rm Spec} (l)}
B$. The surface $\cE_b: X_b\to B$ is semistable and there is a finite, proper map $p_b: X_b \to X$, 
compatible with the map $b: B\to C$. 
\end{theorem}
For the proof see (\cite[Chapter 10]{Liu}). We make the following
observation about Part (ii). Let 
$b\in \cB$, and $w$ be a place of $B$.  Suppose $t'$ maps to $t$, and
the local ring ${\mathcal O}_{B,t'}$ has ramification  degree $d$ over 
${\mathcal O}_{C,t}$. The base changed surface $X^b$ is normal with
$A_d$-singularities at the points on the special fibre at $t'$ which
maps to the singular points of the fibre of $X$ at $t$. The completed
local ring at this singularity is of the form ${\mathcal
  O}_{t'}[[x,y]]/(xy-z^d)$, and the singularity is resolved with
$[d/2]$-blowups 
(\cite[Chapter 10, Lemma 3.21]{Liu}, \cite[Section 2.1.7]{HN}). The surface
$X_b$, the minimal regular model, is the minimal desingularization of
the base changed surface $X^b$.  
In particular, this yields a morphism
$X_b \to X$, compatible with the map $b: B\to C$. 

Suppose $L$ is Galois over $K=k(C)$. The Galois group $G(L/K)$ sits inside
a short exact sequence of the form, 
\begin{equation}\label{seqofgalois}
1\to G(L/lK)\to G(L/K)\to G(l/k)\to 1,
\end{equation}
where the Galois group $G(L/lK)$ can be identified with the automorphism
group $Aut(B/C\times_k {\rm Spec} (l))$. Since $X_b$ is the minimal
desingularization of $X^b$, the action of $G(L/K)$ on
$X^b$ extends to  yield an action of
$G(L/K)$ on $X_b$.

\subsection{ N{\'e}ron-Severi group}
The N{\'e}ron-Severi group of $\bar{X}$, is the group of  divisors on
$\bar{X}$  taken modulo algebraic equivalence.
The  N{\'e}ron-Severi group $NS(X)$ of $X$
is defined to be
the image of ${\rm Pic}_{X/k}(k)$ in $NS(\bar{X})$,
where ${\rm  Pic}_{X/k}$ is the Picard group scheme of $X$ over $k$. 
 The
intersection product of divisors on $\bar{X}$ induces a bilinear
pairing  $<.,.>$ on $NS(\bar{X})$, and hence on $NS(X)$. 
 The intersection product  
of two divisors $D_1, ~D_2$ will  be  denoted by $<D_1,D_2>$ or just $D_1.D_2$.

Suppose $X\to C$ is a semistable elliptic surface, and  $(l, B, b)\in
\cB$. By Part (ii) of Theorem \ref{ssbc}, we obtain  a well-defined
pullback map $p_b^*:~NS(X)\to NS(X_b)$
of the  N{\'e}ron-Severi lattices, satisfying
\begin{equation}\label{pullbackns}
< p_b^*x, p_b^*y>= {\rm deg} (b)<x,y> \quad x, y\in NS( X), 
\end{equation}
where ${\rm deg} (b)$ is the geometric degree of the maps $b$ and  $p_b$.

\subsection{Universal N\'eron-Severi group}
We consider the N\'eron-Severi group of an elliptic surface $X\to C$,
functorially with regard to arbitrary base changes given by finite, separable
maps $B\to C$. The aim is to show that an effective natural transformation
between two such N\'eron-Severi functors, arises from an isomorphism of
the elliptic surfaces.

\begin{definition}
Let $\cE: X\xrightarrow{\pi} C$ be a semistable  elliptic surface defined over a
field $k$. Define the {\em universal N\'eron-Severi group} $UNS(\cE)$ of 
$ \cE$ to be the collection of $(NS(X_b),
<.,.>)$, where $b\in \cB_C$, equipped with the pull back maps
\[ p_a^*: NS(X_b)\to NS(X_{b\circ a}),\]
for any pair of finite morphisms $A\xrightarrow{a} B\xrightarrow{b}
C, ~a\in \cB_B$. 
\end{definition}

\begin{definition}
A {\em (universal) isometry} $\Phi: UNS(\cE)\to UNS(\cE')$ between 
universal N\'eron-Severi groups of  two  semistable elliptic surfaces 
$\cE: X\xrightarrow{\pi} C$ and $\cE': X'\xrightarrow{\pi'} C$ is defined 
to be a collection of
isometries
\[\phi_b: NS(X_b)\to NS(X'_b), \]
indexed by $b\in \cB_C$, such that for any sequence of
finite maps $A\xrightarrow{a} B\xrightarrow{b} C$ with $b\circ a\in
\cB$,  $\phi_{b\circ a}\circ p_a^*=p_a^*\circ \phi_b$, i.e., the following
diagram is commutative: 
\[\begin{CD}
NS(X_{b\circ a}) @>\phi_{b\circ a}>>  NS(X_{b\circ a})\\
@Ap_a^*AA @Ap_a^*AA\\
NS(X_b)  @>\phi_{b}>> NS(X_{b})
\end{CD}
\]
\end{definition}
We also denote $\phi_b$ simply by $\Phi$ if the context is clear.  

\subsubsection{Effective isometries}
We recall that a prime divisor on $X$ (\cite[Chapter II, Section
6]{H}) is a closed integral subscheme of codimension one. A divisor
$D$ on $X$ is said to be effective, if it can be written as a finite,
non-negative integral linear combination of prime divisors. 
Given two elliptic surfaces $X, X'$ over $C$, a map
$\phi: NS(X)\to NS(X')$ is said to be {\em effective}, if it takes the
class of an effective 
divisor on $X$ to the class of  an effective divisor on $X'$.

A (universal) isometry $\Phi$ of universal  N\'eron-Severi 
groups of two semistable elliptic
surfaces is said to be {\em
  effective} if for any $b\in \cB_C$, the isometry 
$\phi_b$ is effective, i.e., it takes the cone of effective divisors
in $NS(X_b)$ to the cone of effective divisors in $NS(X'_b)$. 

\subsection{An effective universal Torelli theorem}
Suppose $\cE: X\xrightarrow{\pi} C$ and $\cE' : X'\xrightarrow{\pi'} C$ are two
semistable elliptic surfaces. By an isomorphism $\theta: \cE \to \cE'$
of the elliptic surfaces, we mean
an isomorphism $\theta: X \to X'$ compatible with
the projections, i.e., $\pi'\circ \theta=\pi$. It is clear that for any 
$b \in \cB_C$, $\theta$ induces an effective isometry 
\[ \theta_b: NS(X_b)\to NS(X'_b).\]
Then $\theta$ induces a universal  effective isometry
$\Theta:  UNS(\cE)\to UNS(\cE')$. 

Our main theorem is the converse, that an effective Torelli holds in
totality considering all base changes 
for elliptic surfaces:

\begin{theorem}\label{uett}
Let $k$ be a field of characteristic zero or finitely generated over
its prime field, and 
$\cE: X\xrightarrow{\pi} C, ~\cE': X'\xrightarrow{\pi'} C$ be
semistable elliptic
surfaces over $k$.

Suppose $\Phi: UNS(\cE)\to UNS(\cE')$ is an 
effective universal isometry between 
universal N\'eron-Severi datum attached to $\cE$ and $\cE'$ as defined
above.  

Then $\Phi$ arises from an unique
isomorphism $\theta: \cE \to \cE'$ between the elliptic surfaces. 

The isomorphism $\theta$ is defined over $k$ if $k$ is finitely generated over
its prime field, and over a quadratic extension $l$ of $k$ in
case $k$ is an arbitrary field of characteristic zero. If we assume
further that the
Kodaira types of the singular fibres of $X\to C$ are of type $I_n$
with $n\geq 3$, then the isomorphism $\theta$ can be  defined over $k$.  
\end{theorem}

This result is the analogue of the refined Torelli theorem for $K3$ surfaces
(\cite[Theorem 11.1]{BHPV}), with the additional assumptions involving 
base changes. 

\begin{example}\label{torellifailure}
 Suppose $X_s$ is a family of non-isotrivial 
elliptic surfaces over a
  curve $C$
  parametrized by a irreducible variety $S$. We assume that the ground
field $k$ is algebraically closed. Let $\eta$ be the generic point of
$S$. For a general point $s\in S(k)$, i.e., outside of a countable
union of proper closed subvarieties of $S$, the specialization map
$i_s: NS(X_{\eta})\to NS(X_s)$ is an isomorphism (\cite[Proposition
3.6]{MP}). By the continuity theorem for interesection products
(\cite[Theorem 10.2]{Fu}), it follows that the specialization map
$i_s$ is an isometry. This gives examples of non-isomorphic elliptic
surfaces whose N\'eron-Severi groups are isometric.

Suppose $k=\C$. The morphism $\pi^*: H^1(C,\Z)\to H^1(\bar{X_s}, \Z)$ 
is an isomorphism preserving the Hodge structures. 
Suppose $h^{02}(X_s)=h^{20}(X_s)=0$,
and the group $H^{1,1}(X_s)=NS(X_s)$. This happens for rational elliptic
surfaces. The rational elliptic surfaces with reduced discriminant 
have a $8$-dimensional moduli (\cite{HL}). 
 In particular, a
Torelli type theorem does not hold for elliptic surfaces in general.   
\end{example}

Thus, in order to obtain Torelli type results, it is necessary to bring in
extra inputs: for example, transcendental inputs like Hodge theory, or
some kind of Galois or universal invariance like we do out here.

\subsection{A (non-effective) universal Torelli theorem}

We now give an analogue of the (weak) Torelli theorem for $K3$ surfaces
(\cite[Corollary 11.2]{BHPV}), where we do not assume that the map
$\Phi$ is effective, but with the additional assumptions involving 
base changes. 

\begin{theorem}\label{utt}
Let $k$ be a field of characteristic zero or finitely generated over
its prime field, and 
$\cE: X\xrightarrow{\pi} C, ~\cE': X'\xrightarrow{\pi'} C$ be
semistable elliptic
surfaces over $k$.

Suppose $\Phi: UNS(\cE)\to UNS(\cE')$ is an  universal isometry between 
universal N\'eron-Severi datum attached to $\cE$ and $\cE'$ as defined
above.  

Then the surfaces $\cE$ and $\cE'$ are isomorphic.
The isomorphism $\theta$ is defined over $k$ if $k$ is finitely generated over
its prime field, and over a quadratic extension $l$ of $k$ in
case $k$ is an arbitrary field of characteristic zero. If we assume
further that the
Kodaira types of the singular fibres of $X\to C$ are of type $I_n$
with $n\geq 3$, then the isomorphism $\theta$ can be  defined over $k$.  
\end{theorem}

\begin{remark} H. Kisilevsky pointed out the relevance of these
  theorems to a conjecture of Y. Zarhin (\cite{K}): suppose $E, ~E'$
  are elliptic curves defined over a number field $k$: if the ranks
  of the Mordell-Weil groups of $E$ and $E'$ are equal over all finite
  extensions of $k$, are $E$ and $E'$ isogenous? 

It would be interesting to know whether analogues of our theorems hold
for elliptic curves defined over number fields. 
\end{remark}

\begin{remark}
It is possible to drop the semistability hypothesis, but instead
require that an isometry of the universal N\'eron-Severi groups exists
whenever both the elliptic surfaces acquire semistable reduction:
\begin{theorem}\label{nonsemistable}
Let $k$ be a field of characteristic zero or finitely generated over
its prime field, and 
$\cE: X\xrightarrow{\pi} C, ~\cE': X'\xrightarrow{\pi'} C$ be elliptic
surfaces over $k$.

Suppose $\Phi: UNS(\cE)\to UNS(\cE')$ is an  universal isometry between 
universal N\'eron-Severi datum attached to $\cE$ and $\cE'$ as defined
above.  

Then the surfaces $\cE$ and $\cE'$ are isomorphic 
over $k$ over a quadratic extension $l$ of $k$ in
case $k$.
\end{theorem} 
The theorem follows from Theorem \ref{utt},
since the property of having semistable reduction is
local. One can produce different curves $B\to C$, 
whose function fields are disjoint and
over which the elliptic surfaces become semistable. By descent, the
isomorphism $\theta$ defined over the various curves $B$ will descend
to an isomorphism between the two elliptic surfaces defined over $C$. 
\end{remark}
\subsection{A reformulation}
Let $\bar{K}_s$ denote a separable algebraic closure of $K=k(C)$
containing $\bar{k}_s$.  Suppose $L$ is a finite extension of $K$
contained in  $\bar{K}_s$. Then $L$ is of the form  $l(B)$, where $l$
is the closure of $k$ inside $L$, and  $B$ is a geometrically 
integral, regular  
projective curve defined over $l$. There is a bijective order reversing
correspondence between 
finite extensions $L$ of $K$ contained in  $\bar{K}_s$ and 
finite, separable maps
 $b:B\to C\times_k {\rm Spec} (l)$, where $B$ is an integral, normal  
projective curve defined over $l$.

Suppose $L/K$ is Galois and $B$ is regular.  
The Galois group $G(L/K)$ acts on $NS(X_b)$. 
Assume now that $X\to C$ is semistable. Consider the direct limit, 
\[ \overline{NS(X)/C}=\varinjlim_{b\in \cB_C} NS(X_b).\]
This acquires an action of the absolute Galois group $G(\bar{K}_s/K)$
of $K$.  Given a sequence of
finite maps $A\xrightarrow{a} B\xrightarrow{b} C$ with $b\circ a\in
\cB$, we have ${\rm deg}(b\circ a)={\rm deg}(b) {\rm deg}(a)$.  Define a 
normalized bilinear pairing on $NS(X_b)$, by 
\[ < x, y>_n= \frac{1}{{\rm deg}(b)}<x,y> \quad x, y\in NS( X_b).\]
With this modified inner product, the map $p_b^*$ gives an isometry of
$NS(X)$ into $NS(X_b)$.  Hence,
the normalized inner products on $NS(X_b)$, 
can be patched to give a
symmetric, bilinear $\Q$-valued pairing, $<.,.> :
\overline{NS(X)}\times \overline{NS(X)}\to \Q$. 
Since $p_b^*$ maps effective divisors to effective divisors, the cone of
effective divisors in $\overline{NS(X)}$ can be defined, and it makes
sense to define effective morphisms between the completed N\'eron-Severi
groups. Theorem \ref{uett} can be reformulated as: 
\begin{theorem} \label{uett-galois}
Let $k$ be a field of characteristic zero or finitely generated over
its prime field, and 
$\cE: X\xrightarrow{\pi} C, ~\cE': X'\xrightarrow{\pi'} C$ be
semistable elliptic
surfaces over $k$.

Suppose $\Phi:  \overline{NS(X)/C} \to \overline{NS(X')/C}$ is an
effective,   $G(\bar{K}_s/k(C))$-equivariant isometry. 

Then $\Phi$  arises from an
isomorphism $\theta: \cE \to \cE'$ between the elliptic surfaces. 
The isomorphism $\theta$ is defined over $k$ if $k$ is finitely generated over
its prime field, and over a quadratic extension $l$ of $k$ in
case $k$ is an arbitrary field of characteristic zero. If we assume
further that the
Kodaira types of the singular fibres of $X\to C$ are of type $I_n$
with $n\geq 3$, then the isomorphism $\theta$ can be  defined over
$k$.  
\end{theorem}

A similar reformulation can be given for Theorem \ref{utt}. 

\subsection{Universal lifts of Picard-Lefschetz isometries}
The question of finding examples of non-effective universal
isometries, leads one  to study 
Picard-Lefschetz reflections.  
Given an element $v\in NS(X)$ with $v^2=-2$, the Picard-Lefschetz
reflection $s_v$ based at $v$ is defined as, 
\[ s_v(D)= D+<D,v>v, \quad D\in NS(X).\]
The Picard-Lefschetz reflection $s_v$ is a reflection around the
hyperplane orthogonal to $v$: 
\[ s_v^2=Id \quad \mbox{and} \quad s_v(v)=-v.\]
The following theorem shows that in the semistable case, the
Picard-Lefschetz reflections can be lifted to universal isometries: 
\begin{theorem}\label{upl}
Let 
$\cE: X\xrightarrow{\pi} C$ be  a semistable elliptic surface over an
field $k$. 
Suppose $x_0\in C(k)$ is an element of
the singular locus $S$  and 
that the fibre of $p$ over
$x_0$ is of type $I_n$ for some $n\geq 3$.  
Let  $v$ be an irreducible component of the fibre $p^{-1}(x_0)$.  Then
there exists a universal isometry $\Phi(v):UNS(\cE)\to UNS(\cE)$, lifting
the Picard-Lefschetz reflections $s_{v}$:
\[ \Phi(v)\mid_{NS(X)}=s_v.\]
\end{theorem}

\subsection{Group of universal isometries}

The proof of Theorem \ref{upl} allows us to determine
 the structure of the group
of universal isometries of a semistable elliptic surface.

We recall that the affine Weyl group of type $\tilde{A}_{n-1}$,
denoted here by $W_n$, is the group with the presentation, 
\begin{equation}\label{genaffineweyl}
 \left< s_0, \cdots,
s_{n-1} \mid (s_is_j)^{m_{ij}}=1\right>,
\end{equation}
where 
\begin{equation}\label{braidrelations}
m_{ii}=1, ~m_{i(i+1)}=3 \quad \mbox{and}\quad m_{ij}=2 \quad
\mbox{for} \quad |i-j|\geq 2.
\end{equation}
Here we are using the notation for a cyclic group of order $n$, and
the obvious meaning for $|i-j|$. 

Restricted to a fibre of Kodaira-N\'eron type $I_n$, 
the Picard-Lefschetz reflections based on the irreducible components
of the fibre generates
the affine Weyl group of type $\tilde{A}_{n-1}$. Suppose that the
irreducible components of the singular fibre are $v_0, v_1,
\cdots, v_{n-1}, v_n=v_0$. The map $s_{v_i}\mapsto s_i$, yields 
an identification of the group
generated by the Picard-Lefschetz reflections to the affine Weyl group
of $\tilde{A}_{n-1}$.

\begin{theorem}\label{guti}
Let $\cE: X\xrightarrow{\pi} C$ be a semistable elliptic surface over a
field $k$ of characteristic zero or finitely generated over its prime
field.  Assume
that the singular locus $S$ is contained in $C(k)$ and let
$S=\{s_1, \cdots,s_r\}$, and that the Kodaira-N\'eron type of the
fibre over $s_i, ~i=1,\cdots, r$ is $I_{n_i}, ~n_i\geq 3$ respectively.

The group $EffAut(UNS(X))$ of universal effective isometries of $\cE$
is a semidirect product $E(k(C))\rtimes Aut(E/k(C))$, where 
$ E(k(C)) $ acts by translations of the section of $\pi$ corresponding
to an element of $ E(k(C)) $. The group $ Aut(E/k(C))$ of
automorphisms of the generic fibre is isomorphic to $\Z/2\Z$.

The group $Aut(UNS(X))$ of universal isometries of $\cE$
is a semidirect product $(\{\pm Id\}\times
\prod_{i=1}^rW_{n_i})\rtimes EffAut(UNS(X))$, 
where $\{\pm Id\}$ is central and is the isomorphism sending every
divisor to its negative; the group $\prod_{i=1}^rW_{n_i}$ is the group
generated by the universal Picard-Lefschetz isometries corresponding
to the irreducible  components of the singular fibres of $\pi$.

\end{theorem}

\subsection{A class of representations of affine Weyl group of type
  $\tilde{A}_{n-1}$}
The process of showing that Picard-Lefschetz reflections based on the
irreducible components of singular fibres of a semistable elliptic
surfaces lift to an universal isometry yields an interesting class of
representations of the affine Weyl group $W_n$. 
Corresponding to a fibration with local ramification degree $e$, we
define a homomorphism, say   $R_e$, of $W_n$ into $W_{ne}$.  

Let $n\geq 3$. Denote by $\V_n$ the vector space
equipped with a symmetric bilinear form and 
basis $v_i, ~i \in \Z/n\Z$ satisfying, 
\[ v_i^2=-2, \quad v_iv_{i+1}=1, \quad v_iv_j=0, \quad i, ~j\in
\Z/n\Z, ~ |i-j|\geq 2,\]
where we are using the obvious meaning for $|i-j|\geq 2$. 
\begin{definition} \label{dfn:vij}
Given a natural number $n$, fix an orientation on $\Z/n\Z$,  for
instance, by identifying $\Z/n\Z$ with the $n$-th roots of unity. 
Given $i, ~j\in \Z/n\Z$, 
define the vector $v(i,j)\in \V_n$, as  
\[v(i,j)=\sum_{k=i}^j v_k, \]
where the indices occuring in the sum are taken with the positive
orientation from $i$ to $j$.  Another way of describing the indices
occuring in the sum is that we take integral representatives for $i$
and $j$ (denoted by the same letter) such that $i<j$
and then the sum goes from $i$ to $j$. It is assumed that the set of
integers $i\leq m\leq j$ maps injectively to $\Z/n\Z$ upon reduction
modulo $n$. 

The {\em support} of $v(i,j)$,
denoted by $spt(v(i,j))$ is
defined to be the set of integers in the interval $[i,j]$. The
{\em length} $l(v(i,j))$ of $v(i,j)$ is
the cardinality of the support of $v(i,j)$.  
\end{definition}
It can be seen that $v(i,j)^2=-2$, and that any two such distinct
vectors are orthogonal provided their supports have a non-empty
intersection (see Lemma \ref{vij}).
For $n\geq 3$,  $k\in \Z/n\Z$ and any natural number $e$,  define the set 
$I(n,e, k)$ to be the collection of vectors of the form $v(i,j)\in
\V_{ne}$ satisfying the following properties: 
\begin{itemize}
\item The support of $v(i,j)$ contains $ek$. 
\item The length of $v(i,j)$ is $e$. 
\end{itemize}
\subsection{Base change}
For the vectors $v_0, \cdots, v_{n-1}$ belonging to $\V_n$, 
define the following vectors $p_b^*(v_k)\in \V_{ne}$, 
\[ p_b^*(v_k)=ew_{ek}+\sum_{0<j<e}(e-j)(w_{ek-j}+w_{ek+j}), \]
where $w_i, ~i\in \Z/ne\Z$ forms the standard basis for $\V_{ne}$. 
This defines a linear map $p_b^*: \V_n\to \V_{ne}$. The significance of
this map is given by Proposition \ref{ssfibrepullback}, 
giving a description of the
inverse images of the irreducible components of a singular fibre under
a base change which is totally ramified of ramification degree $e$ at
the singular point under consideration. 

Define a map $R^e_n:W_n\to W_{ne}$ by defining
on the generators $s_0, \cdots, s_{n-1}$ of $W_n$ as: 
\[ R^e_n(s_k)=\prod_{v(i,j)\in I(n,e,k)}s_{v(i,j)}.\]
The following theorem shows that the maps $R^e_n$ form a system of
representations of the affine Weyl groups, closed under composition: 

\begin{theorem}\label{repnaffineweyl}
Let $n\geq 3$. With notation as above, the following holds: 
\begin{enumerate}
\item For any $v_i, ~i\in \Z/n\Z$,  
\[(R^e_n(s_{v_i}))(p_b^*v)=p_b^*(s_{v_i}(v)). \] 
\item For $e\geq 1$, $R^e_n$ is a representation from $W_n$ to $W_{ne}$. 
\item For any natural numbers $e, ~f$, 
\[ R^f_{ne}\circ R^e_n=R^{ef}_n.\]
\end{enumerate}
\end{theorem}

\section{Action on fibral divisors}\label{sec:fibral}
 We first
characterize fibres by an universal property involving divisibility,
which allows us to show that an universal Torelli isomorphims
preserves fibres upto a sign. Using this, it can be derived that an
universal Torelli isomorphism preserves fibral divisors. 

\subsection{Structure of N\'eron-Severi group of an elliptic surface}
We recall now some well known facts about the N\'eron-Severi group of an
elliptic surface. Under our hypothesis on the $j$-invariant
of $E$, it is known by    N{\'e}ron's  theorem of the base (\cite{LN},
\cite[Theorem 1.2]{Sh}), 
  that $NS(X)$ is a finitely
generated abelian group. For elliptic surfaces, this can also be
proved directly using the cycle class map and that  algebraic and
numerical equivalence coincides on an elliptic surface (\cite[Lecture
VII]{M}, \cite[Section 3]{Sh}). From this last fact,  it also
follows that $NS(X)$ is torsion-free. Further by Hodge index theorem, 
the intersection pairing is a non-degenerate pairing 
on  $NS(\bar{X})\otimes \R$ of signature $(1, \rho-1)$ where $\rho$ is
the rank  of  $NS(\bar{X})$.

Fix a `zero' section $O$ of $\pi:X\to C$.  
Let $T(\bar{X})$ denote the `trivial' sublattice 
of  $NS({\bar{X}})$, i.e., the subgroup generated by the
zero section $O$ and the irreducible components of the fibers of $\pi$. 
By decomposing divisors into `horizontal' sections and `vertical'
fibers, there is an exact sequence 
\[ 0\to T(\bar{X}) \to  NS({\bar{X}}) \to \bar{Q}\to 1.\]
Let $K$ (resp. $K'$) 
 be the function field of $C$ over $K$ (resp. ${\bar {k}}$). Let $E$
be the generic fibre of $\pi~:~X\to C$. This is an elliptic curve
defined over $K$, with
origin $0\in E(K)$  defined by the intersection of the section $O$ with $E$. 

Since $\pi$ is proper, the group $E(K)$ is canonically
identified with the group of sections of $\pi:X\rightarrow C$.  
Denote by
 $(P)$ the image in $X$ of the 
section of $\pi$ corresponding to a rational element
$P ~\epsilon ~E(K)$, and by
by $D(P)$ the divisor $(P)-(0)$ on ${{X}}$. Let 
$T(X)=NS(X)\cap T(\bar{X})$ be
the trivial sublattice of $NS(X)$. We have the following 
description of the Mordell-Weil lattice of the generic fibre due to 
Shioda and Tate, which for lack of a reference, we indicate
the proof over arbitrary base fields: 
\begin{proposition}\label{mwgroup}
The section map $sec: P\mapsto (P) ~\mbox{mod} ~T(X)$ gives an
identification of the Mordell-Weil group $E(K)$ of the generic fibre
$E$ with the quotient group $Q=NS(X)/T(X)$. 
\end{proposition} 
\begin{proof} Over $\bar{k}$, this is Theorem
  1.3 in (\cite{Sh}). Let $G_k$ be the Galois group of $\bar{k}$ over
  $k$. Since ${\rm Pic}(X)(k)={\rm
    Pic}(X)(\bar{k})^{G_k}$, $NS(X)$ defined as the image of ${\rm
    Pic}(X)(k)$ in $NS(\bar{X})$ is $G_k$-invariant. 
Given a divisor $D\in NS(X)$, it can be written
  uniquely in $NS(\bar{X})$ as $(P)+V$, for some $P\in E(K')$ and $V\in
  T(X)$.
  The Galois invariance of $D$ by $G_k$
  implies that $P$ is $G_k$-invariant and hence belongs to
  $E(K)$. Hence the section map $sec$ is surjective. Since $sec$ is
  injective over $\bar{k}$, it  is injective (over $k$), 
and this proves the proposition. 
\end{proof}

\subsection{Euler characteristics.}
We recall some facts about the Euler characteristics of (semistable) elliptic
surfaces.
Let $\chi^t(X)$ denote
the (topological) Euler characteristic of $X$, given by the alternating
sum of the $\ell$-adic betti numbers.  Suppose that $S$ is the ramification
divisor of $\pi$, and the singular fibre at $t\in S$ is of type
$I_{n_t}$. It is known (\cite[Corollary 6.1]{SS}), 
that $12\chi^(X)=\sum_{t\in S} n_t$.

Let $\O_X$ denote the structure sheaf of $X$, 
and $\chi(X)$ the (coherent) Euler characteristic 
$\chi(X)=\sum_{i=0}^2(-1)^i{\rm dim}(H^i(X,\O_X))$. These two Euler
characteristics are related by the formula $\chi^t(X)=12\chi(X)$.
In particular, this shows that $\chi(X)$ is always positive.

Further, if
$(P)$  is any section, then the self-intersection number $(P)^2=-\chi(X)$
(\cite[Corollary 6.9]{SS}).

\subsection{Universal isometries preserve fibres}
In this section, our aim is to show that an universal isometry 
preserves the fibre upto a sign. Since $C$ has a point defined over
$k$, the class of the fibre is in $NS(X)$. 
\begin{definition} Suppose $L$ is a lattice, a finitely generated free
abelian group. An element $l\in L$ is said to be divisible by a natural 
number $d$ if $l\in dL$. 

Equivalently, the coefficients with respect to any integral basis of $L$ 
are divisible by $d$. 
\end{definition}
Given a map $b:B\to C$ of degree $d$, the divisor $p_b^*(F)=dF_b$, 
where $F$ (resp. $F_b$) 
is the divisor corresponding to a fibre of $X\to C$ (resp. $X\to B$).

\begin{proposition}\label{fibre} 
Let $\Phi: UNS(\cE)\to UNS(\cE')$ be a 
universal isometry between 
universal Torelli datum corresponding to two semistable elliptic surfaces 
$\cE: X\xrightarrow{\pi} C$ and $\cE': X'\xrightarrow{\pi'} C$ over a
field $k$. 

Let $F$ (resp. $F'$)  denote the class in the N\'eron-Severi group 
$NS(X)$ (resp. $NS(X')$) corresponding to the fibre $\pi^{-1}(s)$
 (resp.  $\pi'^{-1}(s)$), for some $k$-rational point $s\in C(k)$. 
Then
\[ \Phi(F)=\pm F'.\]
\end{proposition}

\begin{proof} 
The subspace  $[O,F]$ of $NS(X)$ generated by  a section 
$O$ and the fibre $F$ is isomorphic to (\cite[Section 8.6]{SS}),
\[ [O,F]=\begin{cases} [1]\oplus [-1] \quad \text{if $\chi(X)$ is odd}\\
\begin{pmatrix} 0 & 1\\ 1&0 \end{pmatrix} \quad \text{if $\chi(X)$ is even}.
\end{cases}
\]
Since $[O,F]$ is unimodular, there is a direct sum
decomposition, 
\[ NS(X)= [O,F]\oplus [O,F]^{\perp}.\]
From the Hodge index theorem, it follows that $ [O,F]^{\perp}$ is
negative definite. 

Write $ \Phi(F)=eO'+fF'+D',$ where $D'$ is orthogonal to
  both $O'$ and $F'$. Let $b:B\to C$ be a map of degree $d$.
The divisor $p_b^*(F)=dF_b$, 
where $F$ (resp. $F_b$) 
is the divisor corresponding to a fibre of $X\to C$ (resp. $X\to B$).
The pull-back divisor $p_b^*D'$ continues to be
  orthogonal to $O'_b=p_b^*O'$ and the fibre $F'_b$ of $X'_b\to B$,
  since $dF'_b=p_b^*F$. 

  We have $p_b^*F=dF_b$, where $F_b$ is a fibre of $X_b\to B$.
Since $\Phi$ is universal,
the divisor
\[\Phi(p_b^*F)=p_b^*(\Phi(F))=eO'_b+fdF'_b+p_b^*D',\]
is also divisible by $d$ in $NS(X')$. As $d$ can be
arbitrarily chosen, it follows that $e=0$. Then, 
\[0= F^2=\Phi(F^2)= D'^2.\]
The intersection pairing is definite on $[O',F']^{\perp}$. Hence 
$D'=0$, and $\Phi(F)=fF$ for some integer $f$. 

Since $\{O,F\}$ forms a basis of the unimodular subspace $[O,F]$, $F$ 
 can be completed to a basis of $NS(X)$. The map $\Phi$ being
an isometry,  $f=\pm 1$ and this proves the proposition. 
\end{proof}

\subsection{Preservation of fibral divisors}
Our aim now is to show that
$\Phi$ preserves the space of fibral divisors. 

\begin{proposition} \label{fibrepreserve}
Let $\cE: X\xrightarrow{\pi} C$ and 
$\cE: X\xrightarrow{\pi'} C$ be semistable
  elliptic surfaces over $k$, and $\Phi: UNS(\cE)\to UNS(\cE')$
 be an universal isometry between the
  universal Torelli data of $\cE$ and  $\cE'$. 

Then $\Phi$ preserves the space of fibral divisors. 
\end{proposition}
\begin{proof}
By the foregoing lemma, we assume after multiplying by the universal
isometry $-Id$ if
required, that the universal isometry $\Phi$ preserves fibres:
$\Phi(F_b)=F'_b$ for any $b\in \cB$. 
Let $v$ be an irreducible component of a singular fibre. Write, 
\[ \Phi(v)=(P')-(O')+r(O')+V'+sF',\]
where $r, ~s\in \Z$,
$V'$ is a fibral divisor whose support does not contain the zero 
component of any singular fibre i.e., the irreducible
component of the singular fibre intersecting non-trivially the zero section.
We have $(O')V'=FV'=0$. Now, 
\[ \Phi(v)F'= \left((P')-(O')+r(O')+V'+sF'\right)F'=r.\]
Since  \[\Phi(v)F'=\Phi(v)\Phi(F)=\Phi(vF)=0, \]
it follows that $r=0$. 

Let $S'\subset C(k)$ be the
ramification locus of $\pi'$, and for a point $s\in S'$, let $V'_s$ be
$s$-component of $V'$. We now claim that $V'^2+2(P')V'\leq 0$.
For this, it is sufficient to
show that $V_s'^2+2(P').V_s'\leq 0$ for any $s\in S'$.  

Let the fibre of $\pi'$ at $s$ be  of type $I_n$.
 Let $v_0',\cdots, v_{n-1}'$ be the
irreducible components of the fibre at $s$ written in the standard
notation, where $v'_0$ is the component meeting the section $O'$. 
We use cyclic notation (congruence modulo $n$) 
and define $v'_n=v'_0$. 
Write $V_{s}'=a_1v_1'+\cdots+a_{n-1}v_{n-1}'$, for some $a_i\in
\Z$. Here $a_0=a_n=0 $ by assumption on $V'$. 
Suppose  $(P')$ intersects the component  $V_{s}'$ at the component
$v_j$.  Then $2(P')V_{j}'=2a_j$. 

We have, 
\[
V_s'^2=-2\sum_{i=0}^{n-1}a_i^2+2\sum_{i=0}^{n-1}a_ia_{i+1}
=-\sum_{i=0}^{n-1}(a_i-a_{i+1})^2.
\]
The integers $a_1-a_0, \cdots, a_j-a_{j-1}$ and
$a_j-a_{j-1},\cdots, a_{n-1}-a_{n}$ give a partition of $a_j$. 
Hence 
\[ \sum_{i=0}^{j}(a_i-a_{i+1})^2\geq \sum_{i=0}^{j}|a_i-a_{i+1}|\geq |a_j|.\]
Similarly, $\sum_{i=j}^{n}(a_i-a_{i+1})^2\geq |a_j|$. 
Hence, we get that
\[ V_s'^2+2(P')V_s'=-\sum_{i=0}^{n-1}(a_i-a_{i+1})^2+2a_j\leq 0.\]
This proves that $V'^2+2(P')V'\leq 0$.
Suppose now $P'\neq O'$. Then, 
\[\begin{split}
 -2=v^2= \Phi(v)^2& =(P')^2+(O')^2-2(P')(O')+V'^2+2(P')V'-2(O')V'\\
&\leq -2\chi(X').%<-2.
\end{split}
\] 
This yields a contradiction when $\chi(X')>1$, and proves the
proposition in this case. 

Suppose now that $\chi(X')=1$. Consider a base change $b: B\to C$ such that
$\chi(X'_b)>1$. By the above argument, 
the base change of $\Phi(v)= (P')-(O')+V'+sF'$ to
$X'_b$ is fibral. This implies that $(P')=(O')$, i.e., 
there is no sectional component and this proves the proposition. 
\end{proof}

\subsection{Isomorphism of singular fibres}
We now show that a universal Torelli isomorphism preserving fibral
divisors yields an `identification'  of the singular fibres of $X$ and
$X'$. Let $\pi: X\to C$ be a (split) semistable
  elliptic surface over $k$. Suppose  $t\in C(k)$ belongs to the
  singular locus of $\pi$ and the Kodaira type of the fibre is $I_n$. 
For $n\geq 3$, the subgroup $NS(X_t)$ of  $NS(X)$ generated by the irreducible
components of the singular fibre $X_t$ of $\pi$ at $t$, is isomorphic
to the affine root lattice of type $\tilde{A}_{n-1}$. When $n=1$, the
fibre $X_t$ is a nodal curve, and 
$NS(X_t)\simeq \Z$ with trivial intersection pairing. 

\begin{proposition}\label{singlocus}
 Let $\cE: X\xrightarrow{\pi} C$ and 
$\cE: X\xrightarrow{\pi'} C$ be semistable
  elliptic surfaces over $k$, and $\Phi: UNS(\cE)\to UNS(\cE')$
 be an universal isometry between the
  universal Torelli data of $\cE$ and  $\cE'$.  
%Assume further that $\chi(X')>1$.

 Let $S$ (resp. $S'$) be the places of $C(k)$ such
that the fibre $X_t$ (resp. $X'_t$) for $t\in S$ 
(resp. $t\in S'$) is
singular. Then, $S=S'$, and for each $t\in S$ and  
$\Phi$ restricts to an isomorphism $NS(X_t)\to NS(X'_t)$. 
\end{proposition}

\begin{proof} 
Let $N$ be the supremum over the natural numbers $n$ such that the
singular fibres of $\pi$ and $\pi'$ are of type $I_n$. Suppose that 
the singular fibre of $\pi$ (resp. $\pi'$) at $t\in S\cup S'$ is of type 
$I_{m}$ (resp.$I_{m'}$) with  $m, m'\geq 0$.
Choose a
degree $d>N$ morphism $B\to C$ which is totally ramified at $v$ and unramified
at all other points of $S\cup S'$. Let $z\in B(k)$ map to $t$. 
The fibre at $z$ of $\pi$ (resp. $\pi'$) is of Kodaira type $I_{md}$
(resp. $I_{m'd}$),  and at
other points of $S\cup S'$ it remains unchanged. 

By Proposition \ref{fibrepreserve}, $\Phi$ sends fibral divisors to
fibral divisors. 
Since there are no isometries between the root systems $\tilde{A_n}$
and $\tilde{A}_m$ for $n\neq m$, and $d>N$
it follows that $m=m'$. Applying the same argument to 
all points $t\in S\cup S'$ the proposition follows.
\end{proof}

\section{Isogeny of generic fibres}\label{sec:isogeny}
In this section, we show under the hypothesis of Theorem \ref{utt},
that the generic fibres of $\pi$ and $\pi$' are isogenous over a
finite extension of $k(C)$, by invoking the validity
of the Tate isogeny conjecture for elliptic curves under our
hypothesis on $k$. 
 
\begin{proposition} \label{isogeny} Let $k$ be a field of
characteristic zero or finitely generated over its prime field, and
$\cE: X\xrightarrow{\pi} C, ~\cE': X'\xrightarrow{\pi'} C$ be
semistable elliptic surfaces over $k$.  Let $\Phi: UNS(\cE)\to
UNS(\cE')$ be a  universal isometry.

For each rational prime $\ell$ coprime to the characteristic of $k$,
there is an isogeny  $\psi_{\ell} :E\to E'$, defined  over $k(C)$
when $k$ is finitely generated over its prime field  and over
$\bar{k}(C)$ when $k$ is an arbitrary field of characteristic zero,
such that $\psi_{\ell}$ is an isomorphism on the $\ell$-torsion of
$E$.  
\end{proposition}

\begin{proof}
After multiplying by the $-1$ map  sending a divisor $D$ on $X$ to
$-D$, we can assume by Proposition \ref{fibre}, that $\Phi$ maps the fibre
of $\pi$ to that of $\pi'$. Choose a section $O$ (resp. $O'$) of $\pi$
(resp. $\pi'$). Write, 
\[ \Phi(O)=(P')-(O')+r(O')+V',\]
where $V'$ is a fibral divisor.  Since $1=OF=\Phi(O)F'$, we get
$r=1$. Thus the section $O$ gets mapped to a section of $\pi'$ modulo
the lattice spanned by  fibral divisors of  $X'$. Denote this section
by $O'$. 

Using $O$ and $O'$, define the  $G(\bar{K}_s/K)$-invariant `trivial lattices' 
$\overline{T(X)}=\varinjlim_{b\in \cB} T(X_b)$ contained   
 in $\overline{NS(X)}$ (and similarly for $\overline{T(X')}$. 
 From Proposition \ref{mwgroup}, we obtain an
 identification as $G(\bar{K}_s/K)$-modules, 
\begin{equation}\label{mwgggf}
E(K_s)=\overline{NS(X)}/\overline{T(X)}.
\end{equation}

For any natural number $n$, let $E[n]$ denote the group of 
$n$-torsion elements in 
$E(\bar{K}_s)$.
From the identification given by Equation \ref{mwgggf} and by
Proposition \ref{mwgroup},  there is a
$G(\bar{K}_s/K)$-equivariant identification, 
\[ E[n]=\{ D\in \overline{NS(X)}/\overline{T(X)} \mid nD\in
\overline{T(X)}\}
/\overline{T(X)}.\]
Fixing a rational prime $\ell$ coprime to the characteristic of $k$,
we get a compatible system of isomorphisms, $\Phi[\ell^n]:
E[\ell^n]\to E'[\ell^n]$. 
Taking the limit as $n\to \infty$  yields an isomorphism, 
\[ \Phi_{\ell^{\infty}} : T_{\ell}(E)\to T_{\ell}(E'),\]
of the Tate modules of the generic fibres of $X$ and $X'$.

Suppose $k$ is finitely generated over its prime field. By theorems of
Tate, Serre, Zarhin and Faltings (\cite{T, Se, Z, F, FW}), 
establishing the isogeny conjecture of Tate,  
 \[ {\rm Hom}(E, E')\otimes_{\Z} {\Z}_{\ell}\simeq {\rm Hom}_{G_K}(T_{\ell}(E),
T_{\ell}(E')),  \]
there exists a morphism $\psi: E$ to $E'$
defined over $K$, and a scalar $a\in\Z_l$ such that
$a\psi$ corresponds to 
$\Phi_{\ell^{\infty}}$. Here we are using a theorem of Deuring (\cite[Theorem 12]{Cl}) 
that  $\mbox{End}(E)=\Z$, since the 
$j$-invariant $j(E)$ of $E$ is not algebraic over the prime field contained in
$k$. 

The effect of $a\psi$ on $E[\ell^n]$ is given by that of an isogeny,
$\psi_{\ell^n}: \bar{a}_n\psi: E\to E'$, where $\bar{a}$ denotes a
lift to $\Z$ of  the image in $\Z/\ell^n\Z$ of $a\in \Z_{\ell}$. Since
the multiplication by $\ell$ maps from $E[\ell^{n+1}]\to E[\ell^n]$
are surjective and these groups are finite, the map $\psi_{\ell^n}$
coincides with $\Phi[\ell^n]$ on $E[\ell^n]$. 

In particular, for each $\ell$ coprime to the characteristic of $k$,
there is an isogeny  $\psi_{\ell} :E\to E'$, which coincides with the
action of $\Phi[\ell]$ on $E[\ell]$. Since $\Phi[\ell]$ is an
isomorphism, this proves the proposition when $k$ is
finitely generated over its prime field. 
   
Now suppose $k$ is an arbitrary field of characteristic zero. We
assume that $k$ is algebraically closed. The proposition follows now
from the geometric analogue of the Tate isogeny theorem given by
(\cite[Corollaire 4.4.13]{De}).  Choose an embedding of $k$ into the
complex numbers  $\C$. Let $S$ be a finite subset of $C(k)$ containing
the discriminant loci of $\pi$ and $\pi'$. The action of the absolute
Galois group $G(\bar{K}/K)$ acts via the algebraic fundamental group
$\pi_1(C- S, p)$, where $p$ is some chosen basepoint. By
(\cite[Theorem 4.6.10]{Sz}),  the fundamental groups are isomorphic
for base change of algebraically closed fields of characteristic
zero. Further if the elliptic curves $E$ and $E'$ are isogenous over
$\C(C)$, then they are isogenous over some finite extension of $k(C)$
contained inside $\C(C)$, hence defined over $k(C)$ since $k=\bar{k}$.
Hence it is enough to work over $k=\C$. 

The $j$-invariants of $E$ and $E'$ are non-constant elements of
$k(C)$. This property continues to hold for the base change of $E$ and
$E'$ to $\C$. The maps $\pi$  defines an abelian scheme $X-
\pi^{-1}(S) \to C- S$ (and similarly for $\pi'$). By (\cite[Corollaire
4.4.13]{De}), 
\[{\rm Hom}(X- \pi^{-1}(S),X'- \pi'^{-1}(S))\simeq {\rm
Hom}(R_1\pi_*\Z, R_1\pi'_*\Z),\] 
where the left hand side is as morphisms of abelian scheme over $C-S$,
and the right hand side is as
morphisms in the category of locally constant sheaves over $C-S$. By
the universal coefficient theorem for homology, since $H^0$ is
torsion-free, tensoring with $\Z_{\ell}$, we can identify
$R_1\pi_*\Z\otimes \Z_{\ell}\simeq T_{\ell}(E)$ (and similarly for
$E'$), as a module for the absolute Galois group $G_K$ of $K=\C(C)$.
Since tensoring with $\Z_{\ell}$ is fully faithful, we obtain a
$G_K$-equivariant isogeny $\psi: E \to E'$ for each rational
prime $\ell$, defined over $\C(C)$. Arguing as above, proves the proposition.
\end{proof}

\section{Non effective universal Torelli}\label{sec:utt}
In this section we prove Theorem \ref{utt}, 
but over $\bar{k}(C)$ when $k$ is an arbitrary field of characteristic
zero. 

 \begin{proposition}\label{isom}
With the hypothesis of Theorem \ref{utt}, the elliptic surfaces
$\cE$ and $\cE'$ are isomorphic over  $k(C)$
when $k$ is finitely generated over its prime field  and over
$\bar{k}(C)$ when $k$ is an arbitrary field of characteristic zero. 
\end{proposition}

\begin{proof}
Let $L=\bar{k}(C)$ when  $k$ is an arbitrary field of
characteristic zero, 
and equal to $k(C)$ when $k$ is finitely generated over its prime
field. Since the generic fibre uniquely determines the minimal regular
model, it is enough to show that the generic fibres $E$ and $E'$ are
isomorphic over $L$. 
By Proposition \ref{isogeny}, for any rational prime, there is an
isogeny $\psi_{\ell}: E\to E'$,  defined over $L$, 
such that the order of the kernel of $\psi_{\ell}$ is coprime to $\ell$. 

Suppose the kernel of $\psi_{\ell}$ contains a group scheme of the 
form $E[a]$ for some natural number $a$. Since multiplication by $a$
is an isomorphism of $E$ to itself, quotienting by groups of the form
$E[a]$, we can assume that kernel of $\psi_{\ell}$ is cyclic, in that
it does not contain
any subgroup scheme of the form $E[a]$. 

Choose some $\ell$ coprime to the characteristic $p$ of $k$.
Suppose for some $\ell'$ coprime to $p$, the $\ell'$-primary subgroup of
$\mbox{Ker}(\psi_{\ell})$ is non-trivial. 
Consider the isogeny, 
\[ A=\psi_{\ell'}^*\circ \psi_{\ell} : E\to E, \]
where $\psi_{\ell'}^*$ denotes the isogeny $E'\to E$ dual to
$\psi_{\ell'}$. 

Since $\psi_{\ell'}$ has no element of order $\ell'$ in its kernel, so
does $\psi_{\ell'}^*$. Since $\mbox{End}(E)=\Z$ by Deuring's theorem
(\cite[Theorem 12]{Cl}), 
 $A$ is  multiplication
by some integer $a\in \Z$. Thus $\mbox{Ker}(A)=E[a]$. 

On the other hand, the $\ell'$-primary part of  $\mbox{Ker}(A)$ is
isomorphic to  the $\ell'$-primary part of $\mbox{Ker}(\psi_{\ell})$,
and this is not of the  form $E[\ell'^k]$ for any $k$. This yields a
contradiction and implies
that the  $\ell'$-primary part of
$\mbox{Ker}(\psi_{\ell})$ is trivial for any
$\ell'$ coprime to the characteristic of $k$. This also proves the
proposition when characteristic of $k$ is zero. 

When the characteristic of $k$ is $p>0$, the foregoing argument implies
that there is an isogeny $\psi: E\to E'$, such that its kernel $G$
is a finite group scheme of order $p^k$,
not containing any subgroup scheme of the form $E[p]$.  The group scheme 
$E[p^k]$ is a semi-direct product
of the cyclic \'etale group scheme $\Z/p^k\Z$ by the
connected group scheme $E[p^k]^0$.
Suppose $G$ has both an \'etale and
connected component $G^0$. Both $G^0$ and $G/G^0$ will contain subgroup schemes
or order $p$. Then $G$  will contain $E[p]$,
contradicting our assumption on $G$. 
Hence we can assume that the kernel of both $\psi$
and the dual isogeny $\psi^*$ are either \'etale or a connected
group scheme. If $G$ is connected, then
the kernel of the dual isogeny  $\psi^*$ is \'etale as together they make up
$E[p^k]$. Hence we can
assume without loss of generality that $G$ is \'etale, and 
is generated by a section $(P)$, given by a torsion-element of
$P\in E(k(C))$ of order $p^k$. 

At a singular fibre, the group structure of the identity component is
$\G_m$. Hence the section $(P)$ cannot pass through the identity component 
of any singular fibre. Suppose that the singular fibre at $s$ of $\pi$
is of Kodaira type $I_n$ for some natural number $n$. 
By (\cite[Theorem A.1]{DoDo}), applied
inductively, the Kodaira type of the singular fibre of $\pi'$ is
$n/p^k$. But by Proposition \ref{singlocus}, both $X$ and $X'$ have
the same singular fibres. This implies that $k=0$ and the proposition
is proved.   
\end{proof}

\begin{remark} The isomorphism allows us to consider the composition
  of the universal isometry $\Phi$ with itself. This plays a crucial
  role in the proof of Theorem \ref{uett}. 
\end{remark}

\section{Effectivity}\label{sec:effectivity}
We now move towards the proof of Theorem \ref{uett}. In this section,
our aim is to prove the following proposition, giving the consequence
of the effectivity hypothesis that is required for the proof of
Theorem \ref{uett}:  
\begin{proposition}\label{effmaps}
Suppose $\Phi: UNS(\cE)\to UNS(\cE')$ is an effective universal 
isomorphism between universal Torelli 
 data of  two semistable elliptic surfaces as in the hypothesis of
Theorem \ref{uett}. Then for any $b\in \cB_C$, 
$\phi_{b}: ~NS(X_{b})\to
NS(X'_{b})$ sends the 
irreducible components of the singular fibres divisors of $\pi$ to 
the irreducible components of the singular fibres divisors of $\pi'$,
and  sections to sections. 
\end{proposition}

We start with the following lemma characterising fibral divisors: 
\begin{lemma}\label{efffibral}
Let $D$ be an effective divisor on ${X}$. Then $D.F\geq 0$. 
If morever $D.F=0$, then $D$ is a fibral
divisor. 
\end{lemma}
\begin{proof} It is sufficient to prove this over $\bar{k}$, and we
assume now that $k=\bar{k}$.  We can assume that $D$ is an irreducible
closed subvariety of $X$ of codimension one. Suppose that $\pi\mid_D
:D\to C$ is dominant. Since $D$ is a closed subvariety and $\pi$ is
proper,  $\pi\mid_D$ is surjective.  Suppose for some $x\in C(k)$, the
fibre $F_x$ at $x$ is irreducible and $D.F_x=D.F<0$. Then the
fibre $F_x\subset D$. Since this happens for almost all $x\in C({k})$,  the
dimension of $D$ will be $2$, contradicting the fact that $D$ is
of dimension one. Hence $D.F\geq 0$. 

If $D.F=0$, then there exists a point $x\in C({k})$ such that
  $F_x$ and $D$ are distinct effective divisors which are disjoint.
Hence it follows that $\pi\mid_D$ cannot be surjective. This implies
that $D$ is a fibral divisor. 
\end{proof}

\begin{lemma} \label{irr.fibre1}
Suppose $D$ is an irreducible subvariety of ${X}$ and
  $D.F=1$. Then $D$ is a section, i.e.,   $\pi\mid_D :D\to
  C$ is an isomorphism. 
\end{lemma}
\begin{proof}  
Since $D.F=1$, the map $\pi\mid_D :D\to
  C$ is surjective, and of degree $1$. Hence the generic points, say
  $D_0$ of $D$ and $C_0$ of 
  $C$ are isomorphic. Let $s_0:C_0\to D_0$ be an isomorphism. Since
  $\pi$ is proper, this extends to a map $S: C\to D$. The image $s(C)$
is a proper closed subvariety of $D$, and hence is equal to $D$. This
implies that $D$ is a section. 
\end{proof}

\subsection{Characterization of sections and irreducible 
fibre components}

\begin{definition} An effective  divisor $D\in NS({X})$ is said to be
 {\em indecomposable}, if it cannot be written as a sum of two effective
  divisors in $NS({X})$. 
\end{definition}
Note that if $X\to C$ is an elliptic surface with a singular fibre of
Kodaira type $I_n$ with $n\geq 2$, then the fibre is not
indecomposable in $NS(X)$. 

We now characterize sections and fibral divisors: 
\begin{lemma}\label{charsectionsfibral}
Let  $\cE: X\xrightarrow{\pi} C$ be a semistable elliptic surface over a
field $k$.   
Let $D\in NS({X})$, and  $F$ denotes the divisor in $NS({X})$
corresponding to a fibre of $\pi$. Then the following holds:

\begin{enumerate}
\item Suppose $n_t=1$ for all $t\in S$. Then the fibre is
  indecomposable. 

\item A divisor $D$ is an irreducible component of a singular fibre of
  Kodaira type $I_n, ~n\geq 2$ if and
  only if it is effective, indecomposable and  $D.F=0$.

\item A divisor $D$ on $X$ is a section, if and only if it is
effective,  indecomposable and $D.F=1$.  

\end{enumerate}
\end{lemma}
 
\begin{proof} (1) Suppose the fibre is not indecomposable, and is
  written as a non-negative linear combination of effective
  divisors. By Lemma 1, only fibral divisors occur non-trivially in
  such a sum. But since $n_t=1$ for all $t\in S$, the only fibral
  divisors contributing to  $NS(X)$ are multiples of $F$.

(2)  If $D.F=~0$, then $D$ is a fibral divisor. Since any fibral
effective divisor can be written as an integral linear combination of
the irreducible components of singular fibres, if $D$ is
indecomposable then it has to be an irreducible component of a
singular fibre of $X\to C$. 

Suppose $v_0$ is an irreducible component of the singular fibre of
$\pi$ at $t_0$ of Kodaira type $I_{n_{t_0}}, ~n_{t_0}\geq 2$,  
and is not indecomposable. 
Then there is an expression of the
form, 
\[ v_0= \sum_{i\in I}n_i(P_i)+\sum_{t\in S}l_{t,j}v_{t,j}, \quad n_i, ~l_{t,j}\geq 0 \quad
\forall i\in I, ~t\in S, ~0\leq j\leq n_t-1.  \]
wher $S$ is the ramification locus of $\pi$ and for $t\in S$, the
fibre is of Kodaira type $I_{n_t}$. Here $I$ is a finite set and
$(P_i)$ is a section. Intersecting with the fibre implies that $n_i=0$
for all $i\in I$. 

Since the self-intersection of $v_0$ is $-2$, this implies
that $l_{t_0,j_0}\geq 1$, where $v_0=v_{t_0,j_0}$. 
 This implies that there is a sum of the form 
$\sum_{t\in S}l_{t,j}v_{t,j}, \quad l_{t,j}\geq 0$ which is equivalent
to $0$ in $NS(\bar{X})$. Intersecting with the zero section $O$
implies that $l_{t,0}=0$ for all $t\in S$, where $v_{t,0}$ denotes the
irreducible component in the fibre at $t$ meeting the zero section. 
By the theorem of Tate-Shioda (\cite[Corollary 6.13]{SS}),  the rest of the
components are linearly independent and hence $l_{t,j}=0$ for all
$t$ and $j$.  This
implies that the irreducible fibral divisor $v_0$ is
indecomposable.

(3)   If  $D$ is an effective, indecomposable divisor on ${X}$, then
$D$ is irreducible. By Lemma \ref{irr.fibre1}, if $D.F=1$, then $D$ is
a section.

Suppose  the section $(P)$ where $P\in E(K)$, 
 can be written as $(P)=\sum_{i\in I}n_iD_i$, 
where $D_i$ are irreducible subvarieties of $\bar{X}$ and $n_i\geq 0$
for $i\in I$.   By
Lemma \ref{efffibral},
the interesection of any effective divisor with a fibre is
non-negative. Upto reindexing, it can be assumed that there exists an
index denoted $1\in I, ~ D_1.F=1$ and 
$ D_i.F~=~0$ for $i\in I, i\neq 1$. By Lemma \ref{irr.fibre1}, $D_1$
is a section, say $(Q)$, for some $Q\in E(K)$. The function fields of
the generic fibre $E$ and $X$ are isomorphic. Thus at the generic
fibre $P$ and $Q$ are linearly equivalent. Being effective, this
implies that $P=Q$. By Lemma \ref{efffibral}, 
$\sum_{i\in I, i\neq 1}n_iD_i$ is
a non-negative sum of irreducible components of fibres and is linearly
equivalent to zero. By the theorem of Tate-Shioda, $n_i=0$ for all $i\neq 1$, and this proves that $(P)$ is indecomposable. 

\end{proof}

{\em Proof of Proposition \ref{effmaps}.}  By Proposition
\ref{fibre}, an universal effective isometry preserves fibres.  
Hence  Proposition \ref{effmaps}
follows from the above lemmas.

\subsection{Translations}
Given a section corresponding to a rational element $P\in E(K)$, the
translation map, 
$\tau_P : NS(X)\to NS(X)$, given by translating by the section
$(P)$ is an isometry. Further it is effective. 

Suppose $b: B\to C$ is in $\cB_C$. The rational element $P$ can be
considered as an element in $E(L)$, where $L$ is the function field of
$B$ and thus defines a translation isometry from $NS(X_b)$ to
itself. It can be seen that  $p_b^*\circ\tau_P=\tau_{P}\circ p_b^*$. Thus
the collection of translations $\tau_P$ for $b\in \cB_C$ defines an
effective isomorphism of the universal N\'eron-Severi group $UNS(X)$
of $X\to C$. 

Proposition \ref{isom} gives an isomorphism of the elliptic surfaces
$X$ and $X'$. Suppose that under this isomorphism the zero section
$O'$ of $X'\to C$ maps to the section $O''$ of $X\to C$. The 
map $\Phi\circ \tau_{-O''}: UNS(X)\to UNS(X)$ gives an
effective isomorphism of universal Torelli data from the elliptic surface
$\cE$ to itself, preserving  the zero section. 

From now on, we will assume that the universal isometry $\Phi$ as
maps from $UNS(X)$ to itself preserving the zero section.

\section{Revisiting action on fibral divisors}\label{sec:fibral2}
 We state a special, refined version of the universal Torelli
theorem, to take care of both Theorems \ref{uett} and \ref{guti}. 
\begin{theorem}\label{torelli}
Let  $\cE: X\xrightarrow{\pi} C$ be a semistable elliptic surface over
$k$. Let $\Phi: UNS(X)\to UNS(X)$ be an
automorphism of the universal N\'eron-Severi group of $X$ satisfying
the following: 
\begin{itemize}
\item $\Phi$ preserves the fibre: $\Phi(F)=F$.
\item $\Phi$ preserves the zero section: $\Phi((O))=(O)$.
\item $\Phi$ maps the irreducible components of
singular fibres to irreducible components of singular fibres.
\item $\Phi$ sends sections to sections. 
\end{itemize}
Then $\Phi$ arises from either the identity or the
inverse map $\iota: P\mapsto -P$ of the generic fibre $E$ over
$k(C)$. 
\end{theorem}
From what has been done so far, under the hypothesis of Theorem \ref{uett}, the hypothesis of Theorem \ref{torelli} hold true. With a bit of descent, Theorem \ref{uett} will follow from Theorem \ref{torelli}. 

The proof of Theorem will be given in Section \ref{sec:uett}.
In this section, our aim is to show that $\Phi$ is partially geometric,
in that it arises from an
isomorphism of elliptic surfaces restricted to torsion and the fibral
divisors.

\subsection{N\'eron models, torsion elements and the 
narrow Mordell-Weil group} 
We recall some crucial facts that follow from the properties of N\'eron models. Given a point $t\in
S$, let ${\mathcal O}_{t}$ be the local ring of the curve 
$C$ at $t$, and $K_t$ be its quotient field. 
By localization, the elliptic surface
defines an elliptic curve $E_t$ defined over $K_t$.
Let ${\cE}_t$ denote the N\'eron model of $E_t$. This is a group
scheme defined  over  ${\mathcal O}_{t}$, with the property that
$E_t(K_t)={\cE}_t({\mathcal O}_{t})$. The special fibre of the N\'eron
model $\cE_t$ can be identified with the complement of the
singular locus in the fibre $X_t$. In particular, the
collection of connnected components $G_t$ of a singular fibre $X_t$
acquires a group structure,  with the component intersecting the zero
section  as the
identity element of the group law. When the fibre is of Kodaira-N\'eron type
$I_n$, the group of connected components $G_t\simeq \Z/n\Z$.  
The specialization map yields a
homomorphism $\psi: E_t(K_t)\to G_t. $

The main global ingredient in the proof of Theorem \ref{torelli} is the
following theorem (\cite[Corollary
7.5]{SS}), stating that a torsion section is determined by its
intersections with the components of the singular fibres:
\begin{theorem}\label{thm:keyglobal}
The global specialization map yields an 
injective homomorphism,
\[ \psi: E(K)_{tors}\to \prod_{t\in S}G_t, \]
where $E(K)$ is the torsion subgroup of $E(K)$. 
\end{theorem}
 Define the narrow Mordell-Weil group $E_O(K)$ to be the subspace of
$E(K)$ consisting of the elements
$P\in E(K)$ such that the section $(P)$ of $\pi:X\to C$ corresponding
to $P$ intersects each singular fibre  at the identity
component.  Equivalently, $E_O(K)=\mbox{Ker}(\psi)$. A conseqeunce of
Theorem \ref{thm:keyglobal} is that $E_O(K)$ is torsion-free.

\subsection{$\Phi^2$ is partially geometric}
For any $b\in \cB$, let
$\mbox{Tor}({X}_b)$ denote the group of `torsion sections' 
of ${\pi}_b: {X}_b\to {B}$,
corresponding to the torsion elements in the generic fibre
$E({k}(X_b))$ of ${\pi}_b$ of order coprime to the characteristic of
$k$.

The fact that  $\Phi$ can be considered as 
a self-map from  $UNS(X)$ to itself, allows one to compose $\Phi$ with
itself. We have, 
\begin{proposition}\label{square}
 For any $b\in \cB$, the restriction of $\phi_b^2$  to $\mbox{Tor}({X}_b)$ is the identity map. 
\end{proposition}

\begin{proof}
The zero section $O$ of $\pi$ pulls back to the zero section $O_b$ of
$\pi_b$.  The zero section $O_b$ is fixed by $\phi_b$. 
Let $S_b$ be the ramification locus of $\pi_b$.  Suppose the  Kodaira type of the singular fibre at $t\in S_b$  is of type  $I_{m_t}$ for
some $m_t\in \N$. If $m_t\geq 3$, the singular fibre is a chain of
rational curves each with self-intersection $-2$ and intersecting its
neighbours with multiplicity one. Since the component intersecting the
zero section is fixed by $\phi_b$, and $\phi_b$ is an isometry, it
will either act as identity or act as an involution sending the
divisor $v_j$ to $v_{-j}$, where we are using the notation as in 
section 2.9. If $m_t\leq 2$, then there at most two components. 
Thus $\phi_b^2$  acts as identity on each
$G_t$ for each $t\in S_b$. 
The proposition now follows from Theorem \ref{thm:keyglobal}.  
\end{proof}

\subsection{An application of Tate uniformization} We now apply
Tate's uniformization of semistable elliptic curves to gain further control on $\Phi$. Fix a  singular point $s\in S$ where the fibre is of type $I_n$
for some $n\in \N$. Corresponding to $s$, there is a non-trivial
discrete valuation $\nu_s$ of the function field $K$ of $C$, and we 
let $\widehat{K}$
be the completion of the $K$ with respect to the valuation
$\nu_s$. Denote by $L$ the algebraic closure of $\widehat{K}$. 

Since the elliptic surface has semistable reduction at $s$, 
the $p$-adic uniformization theorem of Tate asserts the existence of
$q\in \widehat{K}^*$, such that there is a 
$G(L/\widehat{K})$-equivariant isomorphism, 
\begin{equation}\label{tateunif}
 L^*/q^{\Z}\to E(L).
\end{equation}
Fix a rational prime $\ell$ coprime to the characteristic
of $k$. The $p$-adic uniformization theorem implies that  the Tate
module $T_{\ell}(E/\widehat{K})$ of the elliptic curve $E$ considered over
$\widehat{K}$ sits in the following exact sequence of
$G(L/\widehat{K})$-modules,  
\begin{equation}\label{unifexseq}
 1\to T_{\ell}(G_m)\to T_{\ell}(E/\widehat{K}) \to \Z_{\ell}\to 0. 
\end{equation}
where $T_{\ell}(G_m)=\varprojlim_{n\to \infty} \mu_{\ell^n}$, and
$\mu_{\ell^n}$ is the group of $\ell^n$-th roots of unity in
$L$. 

Let $D_s\subset G(\bar{K}/K)$ 
denote the decomposition group at $s$, defined as the image
under the restriction map to $\bar{K}_s$ of
$G(L/\widehat{K})$. Since $\Phi$ is
$G(\bar{K}/K)$-equivariant, it is also $D_s$-equivariant. 
The Tate module
$T_{\ell}(E)$ is isomorphic as abelian groups to
$T_{\ell}(E/\widehat{K})$. The decomposition group 
$D_s$ preserves the filtration given by Equation
(\ref{unifexseq}). 

With respect to this filtration, it is known that
the Zariski closure of the image of $D_s$ inside $Aut(V_{\ell}(E))$
contains the group of matrices which act as identity on the associated
graded decomposition of the space
$V_{\ell}(E)=T_{\ell}(E)\otimes_{\Z_{\ell}}\Q_{\ell}$ (\cite{Se}). Choosing an
appropriate basis, the Zariski closure of the 
image thus contains the subgroup $U$ of 
upper triangular unipotent matrices. Since $\Phi$ is equivariant with respect to the action of $D_s$, it is equivariant with respect to $U$. Hence we have, 
\begin{lemma}
  The map $\Phi$ acting on $V_{\ell}(E)$ is upper triangular with respect to the
  filtration given by Equation \ref{unifexseq}.
\end{lemma} 

Since the only upper triangular matrices of order $2$ are diagonal
matrices with entries $\{\pm 1\}$ along the diagonal, 
combining this with Proposition \ref{square}, we get

\begin{corollary}\label{pm}
Suppose $n\geq 1$. Let $P$ be a generator for the group
$\mu_{\ell^n}(L)\subset E[\ell^n]$, and  $Q$ be a
generator for  the
quotient group $E[\ell^n]/\mu_{\ell^n}(L^*)$. Then 
\[ \Phi((P))= \pm (P) \quad \mbox{and }\quad  \Phi((Q))= \pm (Q).\]
\end{corollary} 

\subsection{$\Phi$ is geometric on torsion}
Corollary \ref{pm} allows us to conclude that $\Phi$ is geometric restricted to
$\ell^{\infty}$-torsion:
\begin{proposition}\label{prop:geomtors}
With hypothesis as in Theorem \ref{torelli}, 
$\Phi$ restricted to  $E[\ell^{\infty}]$ is either identity or the
inverse map $P\mapsto -P$, where $\ell$ is a rational prime coprime to
the characteristic of $k$. 
\end{proposition}
\begin{proof}%[Proof of Proposition \ref{prop:geomtors}]
Suppose for $\ell^n$ and $P, Q\in E[\ell^n]$ as in Corollary \ref{pm}, 
\[ \Phi(P)= P \quad \mbox{and }\quad  \Phi(Q)= -Q.\]
Now $P$ and $Q+P$ also satisfy the hypothesis of Corollary
\ref{pm}. The foregoing equation yields, $\Phi(Q+P)=-Q+P $. By
Corollary \ref{pm}, $\Phi(Q+P)=-Q+P $ is equal to either $Q+P$ or
$-(Q+P)$. This implies respectively, $2Q=0$ or $2P=0$. This yields a
contradiction if $\ell^n>2$. A similar argument works when
$\Phi(P)=-P$. 
\end{proof}

Hence restricted to $ E[\ell^n]$ and for any $n$ uniformly, 
we have that $\Phi$ is either identity or the
additive inverse map. After multiplying the base change Torelli
isomorphism $\Phi$ by the morphism induced by the $-Id$ isomorphism of
the elliptic surface, we can assume that $\Phi$ induces the identity
map on  $ E[\ell^n]$ and for any $n$.

\subsection{$\Phi$ is geometric on fibres}

\begin{proposition} \label{geomonfibres}
 With hypothesis as in Theorem \ref{torelli},
 upto multiplication by an element of
  $Aut(E/k(C))\simeq \Z/2\Z$, $\phi_b$ acts as identity on the trivial lattice
$T(X_b)$ for any $b\in \cB$. 
 \end{proposition}

\begin{proof}
 Let $\ell$ be a rational prime coprime to
the characteristic of $k$. By Proposition \ref{prop:geomtors}, we can assume that upto
multipliying  by an element of
  $Aut(E/k(C))$, $\Phi$ acts trivially on $E[\ell^{\infty}]$. We need
  to conclude that $\Phi$ acts trivially on the fibral divisors.  

For this, it is enough to prove it for some base change $b$, since
$p_b^*$ is injective. The injectivity of $p_b^*$ follows from the fact that the intersection pairing on the N\'{e}ron-Severi group of an elliptic surface is non-degenerate taken in conjunction with Equation \ref{pullbackns}.
Consider the base change over which the elements
of $E[\ell]\subset E(l(B))$. It follows from the exact sequence
(\ref{unifexseq}), that
for any singular fibre of $\pi_b$, there will be a $\ell$-torsion
section, not of order $2$ in the group $E[\ell]/\mu_{ell}$, where
$\mu_{\ell}$ is sitting in $E[\ell]$ by Tate uniformization. Since
$\Phi$ acts trivially on $E[\ell]$, and $\Phi$ respects the intersection
product, it follows as in the proof of Proposition  \ref{square} that
 $\Phi$ acts as identity on the irreducible components of any singular
 fibre. 
\end{proof}

\section{Proof of Theorem \ref{uett}}\label{sec:uett}
In this section, we give a proof of Theorem \ref{torelli} (and  
Theorem \ref{uett}). By Proposition
\ref{isom}, the elliptic surfaces $X\to C$ and $X'\to C$ become
isomorphic (over a possibly quadratic extension of $k(C)$ contained
inside $\bar{k}(C)$ in case the characteristic of $k$ is
zero). Assume now that the elliptic surfaces are isomorphic. By
Corollary \ref{effmaps},  the map $\Phi$ preserves sections and the 
irreducible components of the singular fibres divisors of $\pi$. 
Translating by a section, we can assume that $\Phi$ preserves the zero
section of $\pi$. Finally by Proposition \ref{geomonfibres}, 
upto multiplication by an element of
  $Aut(E/k(C))$, we can assume that $\Phi$ acts as identity on the
  trivial lattice. 
 Thus the proof of Theorem \ref{torelli} follows from
the proof of the following theorem: 

\begin{theorem} \label{idonfibres} With hypothesis as in Theorem
  \ref{torelli}, assume further that $\phi_b$ acts as identity on the
  trivial lattice $T(X_b)$ for any $b: B\to C$. Then $\Phi$ is the
  identity map.  
\end{theorem}

\begin{proof} 
For $(l, B, b)\in \cB_C$, let $K_b$ denote the function field $l(B)$. 
Since $\Phi$ fixes the trivial lattice, by passing to
  the quotient $NS(X_b)/T(X_b)$, it yields a homorphism, say $\phi_b^0: E(K)
  \to E(K)$ of the 
  Mordell-Weil group $E(K_b)$ of the generic fibre to itself. 
For $P\in E(K_b)$, let $u_b(P)=\phi^0_b(P)-P$. The map $u_b: E(K_b)\to E(K_b)$ is
a homomorphism,  $ u_b(P+Q)=u_b(P)+u_b(Q)$. The universal property of
$\Phi$ implies that the maps $u_b$ patch to give a map $
U: E(\bar{K})\to E(\bar{K})$.

\begin{lemma} For any $P\in E(K_b)$, $u_b(P)$ lies in the narrow
  Mordell-Weil group $E_O(K_b)$. 
\end{lemma}
\begin{proof}
 The  group of components
of the special  fibre ${\mathcal E_{b,t}}$ of the N\'eron model of the
base changed elliptic curve $E_b$ at a  point $t$ of ramification of
$\pi$ is
indexed by the irreducible components of $X_t$. Since by assumption,
$\Phi$ acts as identity on the set of irreducible components of the
singular fibre, the sections $\Phi((P))$ and $(P)$ pass through the same
irreducible component of the singular fibre. By the group law on
${\mathcal E}_{b,t}$, it follows that the element $u_b(P)$ passes
through the identity component of $X_t$, and this proves the lemma.
\end{proof}

Next,  we observe the following Galois invariance property: 

\begin{proposition} Let  $L=k(B)$ be a finite Galois extension of
  $K=k(C)$ for some $b\in {\mathcal B}_C$. 
 Suppose $u\in E_O(K), ~u'\in E_O(L)$ and $nu'=u$
  for some $n$ coprime to $p$. 
 Then $u'\in E_O(K)$, i.e., $u'$ is
  $G(L/K)$-invariant. 
\end{proposition}
\begin{proof}
The section $(O)$ is defined over $K$. Thus the identity components of
the singular fibres of $X_b$ are invariant by
$\rm{Gal}(L/K)$.  Now for any $s$ in the ramification locus of
$\pi_b$, 
\[ \sigma(u').v_s^0= \sigma(u').\sigma(v_s^0)=u'.v_s^0=1.\]
Hence we obtain that $\sigma(u')\in E_O(L)$. Since
$n\sigma(u')=\sigma(nu')=\sigma(u)=u$, we have $\sigma(u')=u'+t$,
for some $n$-torsion element $t\in E(L)$. Since $u'\in E_O(L)$, the
element $t\in E_O(L)$. Since $E_O(L)$ is torsion-free, this implies
$t=O$ and proves the proposition. 
\end{proof}

\begin{corollary} Given any $0\neq u\in E_0(K)$, there exists some $n$
  sufficiently large such that any $u'\in E(\bar{K})$ with $nu'=u$,
  then $u'$ cannot lie in $E_0(\bar{K})$.
\end{corollary}
If $u'\in E_0(\bar{K})$, by the proposition, we have that $u'\in E(K)$, and
hence in $E_O(K)$. 
By the theorems of Mordell-Weil (\cite{Si1}), and Lang-N\'{e}ron (\cite{LN}),
$E_O(K)$ is a finitely generated abelian group,  and free by 
Theorem \ref{thm:keyglobal}. Hence the corollary follows.  

We can now finish the proof of Theorem \ref{idonfibres}:   
Given $P\in E(K)$ and $u(P)\in E_O(K)$, choose $n$ as
in the above corollary, and $Q\in E(\bar{K})$ with $nQ=P$.  
Then $nU(Q)=U(P)$. Since $U(Q)$
belongs to $E_0(\bar{K})$, the corollary implies that $U(Q)=0$, i.e.,
$\phi^0(Q)=Q$. Since $nQ=P$, and sections are mapped to sections, 
 this implies $\Phi((P))=(P)$. 
\end{proof}

\subsection{Descent and proof of Theorems \ref{uett} and \ref{utt}}
The preceding arguments establish Theorem
\ref{uett} and \ref{utt}, except in the case when the characteristic
of $k$ is zero and the Kodaira types of the singular fibres are of the
form $I_n$ with $n\geq 3$. Here we have a priori that $\Phi$ arises
from an isomorphism $\theta: \cE \to \cE'$ which is defined over a
quadratic extension $l$  of $k$.  Let $\sigma$ be the non-trivial
element of ${\rm Gal}(l/k)$. Denote by $\Phi^{\sigma}=\sigma\circ
\Phi\circ \sigma$. These maps are equal on the N\'eron-Severi groups
 $N(X_t)$ of the singular fibres $X_t, ~t\in S$ of $X\to C$. 
It follows from Proposition
\ref{psitoPsi}, that the maps $\Phi$ and $\Phi^{\sigma}$ are equal. By
Theorem \ref{idonfibres}, it follows that the maps $\theta$ and
$\theta^{\sigma}$ are equal on $\ell^{\infty}$-torsion, and hence they
are equal. This establishes the required descent property for the proofs  of Theorems \ref{uett} and \ref{utt}. 

\section{Base change and fibral divisors}\label{sec:basechange}
Let $\cE: X\xrightarrow{p} C$ be a semistable elliptic surface over
$k$. Fix a point $x_0$ of in the ramification locus of $\pi$ 
such that the fibre of $\pi$ over 
$x_0$ is of type $I_n$ for some $n\geq 2$.  
Let $v_i$   be the
irreducible components of the singular fibre of $p$ at $x_0$
where the indexing is the group $\Z/n\Z$, 
and the intersection multiplicities are as follows for  $i, j \in \Z/n\Z$: 
\[ v_i^2=-2, \quad v_iv_{i+1}=1\quad\mbox{and} \quad v_iv_j=0 \quad
\mbox{for}~ j\neq i, i-1, i+1. \]
We assume that the zero section passes through  $v_0$. 

 Let  $b\in {\mathcal B}_C$ be  a separable, finite morphism 
of degree $d$.  Fix a point 
$y_0\in B(k)$ mapping to $x_0$. Assume that $y_0$ is 
totally ramified over $x_0$ of degree $d$. It is
known that the fibre of $X_b\to B$ over the point $y_0$ is of type
$I_{nd}$.  Let $w_i, ~i \in \Z/nd\Z$ be the
irreducible components of the singular fibre over $y_0$, with indexing
similar to the one given above. 
We would like to describe the inverse image divisors $p_b^*(v_i)$. 
\begin{proposition}\label{ssfibrepullback}
With notation as above, 
\[ p_b^*(v_i)=dw_{di}+\sum_{0<j<d}(d-j)(w_{di-j}+w_{di+j}).\]
\end{proposition}
\begin{proof} Since $b$ is totally
  ramified at $y_0$ of degree $d$, $p_b^*F_{x_0}=dF_{y_0}$, where
  $F_{x_0}$ (resp. $F_{y_0}$) denotes the fibre of $p$ at $x_0$
  (resp. the fibre of $p_b$ at $y_0$).

For $i\in \Z/d\Z$, let $w_{k_i}$ denote the `inverse
  image divisor', i.e., the strict transform in $X_b$ of the inverse
  image   of $v_i$ in $X^b=X\times_C B$.   The multiplicity of $w_{k_i}$ 
in $p_b^*(v_i)$ is $d$. Further $p_b^*(v_i)$ is an effective divisor, and 
\[ dF_{y_0}= p_b^*F_{x_0}=\sum_{i\in \Z/d\Z}p_b^*(v_i).\]
Hence, 
\[ p_b^*(v_0) =dw_0 +\sum_{i\neq 0} a_iw_i, \quad \mbox{where}\quad 0\leq a_i\leq d.\]
Since $w_{k_j}$ occurs in $p_b^*(v_j)$ with multiplicity $d$, it follows
that $a_{k_j}=0$ for $j\neq 0$. 

The map $p_b: X_b\to X$ is a finite proper map.
The divisors $w_j$ for $j\neq k_i$ for
any $i$, map to a point under $p_b$ and are the exceptional divisors in $F_{y_0}$.  By the projection formula of intersection theory (\cite[Appendix A, p. 427]{H}),
\[p_b^*(w_jp_b^*(v_0))=p_b^*(w_j)v_0=0.\]
Since $p_b$ is finite and proper, it follows that $w_jp_b^*(v_0)=0$. Thus, 
\begin{equation}\label{eqn:average}
 2a_j= a_{j-1}+a_{j+1}, \quad  j\neq k_i \quad \mbox{for} \quad i\in \Z/n\Z.
\end{equation}
Since the multiplicities $a_j$ are bounded by $d$, if the multiplicity
$a_j=d$ for some $j$,then the neighbouring multiplicities $a_{j-1}$ and $a_{j+1}$ are also equal to $d$.  But $a_{k_i}=0$ for $i\neq 0$, and this
implies that the multiplicity of any exceptional divisor $w_j$
occuring in $p_b^*(v_0)$ is strictly less than $d$. 

Now, $(p_b^*v_0)^2=dv_0^2=-2d$. Since the exceptional divisors
interesect trivially with $p_b^*(v_0)$, and the inverse image divisors
$w_{k_j}$ have multiplicity zero in $p_b^*(v_0)$ for $j\neq 0$, we
obtain
\[ d\left( (w_0,dw_0) +a_1(w_0, w_1) +a_{-1}(w_0, w_{-1})\right)=-2d.\]
This implies, $a_1+a_{-1}= 2d-2$, and this is possible if and only if 
$a_1=a_{-1}= d-1$. From Equation (\ref{eqn:average}), it follows that 
for $0< i\leq d$, $a_i=a_{-i}=(d-i)$. 

Thus, $p_b^*v_0= D_1+D_2$, where 
$D_1=dw_{0}+\sum_{0<j<d}(d-j)(w_{-j}+w_{j})$, and the multiplicity
of the components $w_i$ for $|i|\leq d$ is zero in $D_2$. Hence 
the divisors $D_1$ and $D_2$ are orthogonal. The divisors 
$D_1$ and $D_2$ can be considered in the
negative definite space spanned  by the divisors
$w_0, \cdots, w_{d-1}, w_{d+1}, \cdots w_{nd-1}$. Since $D_1^2=-2d=(p_b^*v_0)^2$, this
implies $D_2=0$. Hence we obtain that 
\[p_b^*v_0 =dw_{0}+\sum_{0<j<d}(d-j)(w_{-j}+w_{j}),\]
and a similar expression holds for each $p_b^*v_i$:
\[p_b^*v_i=dw_{k_i}+\sum_{0<j<d}(d-j)(w_{k_i-j}+w_{k_i+j}).\]
Since the sum of these divisors is equal to $dF_{y_0}$, it follows
that $k_i$ are multiples of $d$. Since $p_b^*v_i$ and $p_b^*v_{i+1}$ have
non-zero intersection, we get $k_i=di$ for $i=0, \cdots, d$, and this
proves the proposition.  
\end{proof}

\section{A representation of the affine Weyl group \\ 
of $\tilde{A}_{n-1}$:
  Proof of Theorem \ref{repnaffineweyl}} \label{sec:repn}
Our aim is to show that the Picard-Lefschetz isometries based at
irreducible components of fibral divisors lift to universal
isometries. We first work out the underlying representation
theoretical aspect, which arise when we consider the action of the
reflections on the subspace of the N\'eron-Severi group contributed by
the components of a fibre. 

Corresponding to a fibration with local ramification degree $e$, we
define a homomorphism, say   $R_e$, of $W_n$ into $W_{ne}$.   We define
this representation on the generators, and verify the braid relations
are satisfied. We also need to check that this representation gives
the lift of the Picard-Lefschetz reflections to an universal
isometry. We first work out some of the linear algebra
considering the vectors $v(i,j)$ defined as in Definition
\ref{dfn:vij}. We use the notation from Section \ref{sec:theorems}. 

\begin{lemma}\label{vij}
With notation as in Definition (\ref{dfn:vij}), the following holds: 
\begin{enumerate}
\item $v(i,j)^2=-2$. In particular, the transformation, 
\[ s_{v(i,j)} (v)= v+(v(i,j),v)v(i,j),\quad v\in \V_n,\]
defines a reflection.
\item Suppose $v(i_1, j_1)\neq v(i_2, j_2)$ are distinct vectors of 
 equal length such that the intersection of their supports is
 non-empty. Then they are orthogonal. Equivalently the reflections  
$s_{v(i_1,j_1)}$ and $s_{v(i_2,j_2)}$ commute. 
\item Any vector $v$ with $v^2=-2$ is of the form $v(i,j)+rF$, where 
$F=\sum_{i\in \Z/n\Z}v_i$ is the `fibre' and $r\in \Z$. 
\item The following relation is satisfied by the reflections
  $s_{v(i,j)}$:
\begin{equation}\label{indeqnsvij}
s_{v(i,j)}=s_{v_j}s_{v(i,j-1)}s_{v_j}.
\end{equation}
\end{enumerate}
\end{lemma}
\begin{proof} For the proof of (1),
\begin{align*}
v(i,j)^2&=\sum_{k=i}^j
v_k^2+v_iv_{i+1}+v_{j}v_{j-1}+\sum_{k=i+1}^{j-1}(v_kv_{k-1}
+v_kv_{k+1})\\
&=-2(j-i+1)+2+2(j-i-1)=-2.
\end{align*}
To prove (2), let $(i_1,j_1)=(i,j)$, and $i_2=i+k$ with $k>0$.
 The hypothesis 
imply that $k\leq (j-i)$, and $j_2=j+k$. The inner product, 
\begin{align*}
v(i,j)v(i+k,j+k)&=v_{i+k-1}v_{i+k}+v(i+k,j)^2+v_jv_{j+1}\\
& =1-2+1=0.
\end{align*}
To prove (3), write $v=\sum_ia_iv_i+rF$, where we can assume by
absorbing into the fibre component, that not all $a_i$ have the same
sign. Write $v=v_+-v_{-}$, where $v_+=\sum_{a_i>0}a_iv_i$ and 
$v_{-}=-\sum_{a_i\leq 0}a_iv_i$. Then, 
\[ v^2=v_+^2+v_{-}^2-2v_+v_{-}.\] 
If both $v_+$ and $v_{-}$ are non-zero, then $v^2<-2$. Since $v(i,j)+v(j,i)=F$, by working with $-v$ if required,
we can assume that $v_{-}=0$. 
Assume now that the support of $v$, the set of indices $i$
such that $a_i>0$ is an interval of the form
$[k,l], ~0\leq k <l <n$. Write  $v=v(k,l)+v'$, where
$v'=\sum_{k\leq i\leq l}b_iv_i,~b_i\geq 0$. Now,
\begin{align*}
  v^2&=v(k,l)^2+2v(k,l)v'+v'^{2}\\
  &=-2-2\sum_{k\leq i\leq l}b_i+b_k+b_l+2\sum_{k<i<l}b_i+v'^{2}\\
  &=-2-b_k-b_l+v'^2.
\end{align*}
Hence if $v^2=-2$, then $v=v(k,l)$.
This equation also implies that for $v$ as above $v^2\leq -2$.
Thus if $v^2=-2$, then $v$ cannot be written as a sum of two vectors with disjoint support. This proves Part (3). 

For (4), the vectors $v(i,j-1)$ and $v_j$ generate a two
dimensional non-degenerate subspace. On the orthogonal complement of
this subspace, both the reflections act as identity, and hence 
 it suffices to verify the formula
\[ s_{v(i,j)}=s_{v_j}s_{v(i,j-1)}s_{v_j},\]
in the two dimensional situation. Since $v(i,j-1)$ and $v_j$ are
vectors with self-intersection $-2$ and $v(i,j-1).v_j=1$, this is classical.
\end{proof}

\subsection{Cyclic permutations}
Before we proceed with the proof of Theorem
\ref{repnaffineweyl}, it is convenient to represent the transformations 
in terms of standard permutation symbols $(j_1\cdots j_k)$.
We consider permutations on the cyclic set $\Z/n\Z$. 
Assign the permuatation $s_{v_i}\mapsto (i, i+1)$. This gives a
representation of the affine Weyl group $W_n$ as permutations on the
cyclic set $\Z/n\Z$. It is clear that the braid relations are
satisfied.

We observe that dropping one of the generators, 
for example $s_0$ from the generators of
the affine Weyl group, the remaining generators $s_1, \cdots, s_{n-1}$ 
satisfy the braid relations defining the symmetric group $S_{n+1}$ (the permutation group on $(n+1)$-symbols) on
$n$-generators:  
\[
m_{ii}=1, ~m_{i(i+1)}=3 \quad \mbox{and}\quad m_{ij}=2 \quad
\mbox{for} \quad |i-j|\geq 2, \quad i, j\in \{1, 2, \cdots, n\}.
\]
The above permutation representation restricted to any of the
symmetric groups obtained by omitting one of the generators $s_i$ is injective, since the group generated is not
of the form $\Z/2\Z$ if the number of generators is at least two.
\begin{definition} 
An expression (or an equation) in the free group 
on the generators of the affine Weyl
group $W_n$ (or in $W_n$) 
is said to be {\em local}, if it does not involve all the
generators in a non-trivial manner. 
\end{definition}
When the equation or expression is local, then to check its properties
or the validity of the equation in the affine Weyl group, it is
sufficient to work with the (local) symmetric group generated by the
generating reflections of the affine Weyl group involved in the
equality. In such a case, it is sufficient to work within the group of cyclic permutations. 

We now recall the definition of the representation 
$R^e_n:W_n\to W_{ne}$. To keep track of the difference, the standard basis of $\V_n$ is given by $v_i, ~i\in \Z/n\Z$ and that of  $\V_{ne}$
is given by $w_i, ~i\in \Z/ne\Z$. The representation  $R^e_n$
 is defined
on the generators $s_{v_0}, \cdots, s_{v_{n-1}}$ of $W_n$ as: 
\begin{equation}\label{defnrepn}
 R^e_n(s_{v_k})=\prod_{w(i,j)\in I(n,e,k)}s_{w(i,j)},
\end{equation}
where the set 
$I(n,e, k)$ is the collection of vectors of the form
$w(i,j)\in \V_{ne}$ of length $e$ and support containing $ek$. 

For the proof, especially of parts (2) and (3) of Theorem  \ref{repnaffineweyl}, we observe that the various statements
are local in the above sense, that it is enough to work with the
symmetric group and hence enough to work with the permutations. This
is because if $e>1$, then the elements
$R^e_n(s_{v_j})$ are in some appropriate symmetric groups by Part (4)
of Lemma \ref{vij}.

\subsection{An inductive definition}
We first need to
check that $ R^e_n(s_{v_k})$ can be written in terms of the generators
$s_{w_i}$. The following inductive definition of  $ R^e_n(s_{v_k})$ is
arrived at trying to ensure that the lifts satisfy Part (1) of
Theorem \ref{repnaffineweyl}, of being compatible with the base
change map on the N\'eron-Severi groups given by Proposition
\ref{ssfibrepullback}. 

For $0\leq j <n, ~1\leq i\leq e$, 
define the following isometries of $\V_{ne}$: 
\begin{align}
T(v_j,0)& =U(v_j,0)=s_{w_{je}}\\
U(v_j,i)& =s_{w_{je-i}}s_{w_{je+i}}\\
T(v_j,i)&=U(v_j,0)\cdots U(v_j,i)=T(v_j,i-1)U(v_j,i)\\
S(v_j,i)& =T(v_j,i-1)\cdots T(v_j,0)=T(v_j,i-1)S(v_j,i-1).
\end{align}

Since $n\geq 3$ and $e\geq 1$, an inductive argument implies that for
any given $j$ and $i$, the expressions $U(v_j,i), T(v_j,i)$ and
$S(v_j,i)$ are all local. 

 \begin{lemma}\label{permutnot}
With the assignment $s_{w_i}\mapsto (i, i+1)$, 
\begin{align}
T(v_j,k)&=(je, je-1, \cdots, je-k, 1,\cdots, je+k+1)\\
S(v_j,k)&=(je-k+1,je+1)(je-k+2,je+2)\cdots (je,je+k).
\end{align}
\end{lemma}
\begin{proof}
The proof is by induction on $k$, and we carry it out for $j=0$.
 The transformation 
$U(v_0,k)$ is given in permutation notation as, 
\[ U(v_0, k)=(k, k+1)(-k, -k+1).\]
Assuming that the lemma has been proved for $k-1$. Then
\begin{align*}
T(v_0,k)& =T(v_0,k-1)U(v_0,k)\\
&=(0, -1, \cdots, -k+1, 1,\cdots, k)(k, k+1)(-k, -k+1)\\
&=(0, -1, \cdots, -k, 1,\cdots, k+1).
\end{align*}
Similarly, assuming that the proposition holds for $S(v_0,k-1)$, 
\begin{align*}
S(v_0,k)& =T(v_0,k-1)S(v_0,k-1)\\
&=(0, -1, \cdots, -k+1, 1,\cdots, k)(-k+2,1)(-k+3,2)\cdots (0,k-1)\\
&=(-k+1,1)(-k+2,2)\cdots (0,k).
\end{align*}
This proves the lemma. 
\end{proof}

\begin{lemma}
In $W_{ne}$, the reflection
 $s_{w(p,q)},~ p<q, ~|q-p|\leq e$ corresponding to the vector $w(p,q)=w_p+w_{p+1}+\cdots +w_{q}$ is local.  The cyclic permutation corresponding to  $s_{w(p,q)}$ is the permutation $(p,q+1)$. 
\end{lemma}
\begin{proof} The proof is by induction on $|q-p|$. The case $q=p$
  follows from the definition. By Part (4) of
  Lemma \ref{vij}, 
\[ s_{w(p,q)}=s_{w_q}s_{w(p,q-1)}s_{w_q}. \]
This implies the locality of  $s_{w(p,q)}$ in the given range. Translating
to the permutation notation, we see that 
\[ (q, q+1)(p,q)(q,q+1)=(p,q+1).\]
This proves the lemma. 
\end{proof}

Combining the two foregoing lemmas and the definition of $R^e_n(s_{v_j})$,
we have the following corollary, 
\begin{corollary}
  For $0\leq j <n$, the permutation realization of  $R^e_n(s_{v_j})$ is
  \[((j-1)e+1,je+1)((j-1)e+2,je+2)\cdots (je,(j+1)e).\]
 In particular, $R^e_n(s_{v_j})\in W_{ne}$.  
\end{corollary}
\begin{proof}
  By definition,
  \[  R^e_n(s_{v_j})=s_{w((j-1)e+1,je)}s_{w((j-1)e+2,je+1)}\cdots s_{w(je,(j+1)e-1)}.\]
The permutation realization of the right hand side is nothing more than   
\[S(v_j, e)=((j-1)e+1,je+1)((j-1)e+2,je+2)\cdots (je,(j+1)e).\]
\end{proof}

\begin{proof}[Proof of Part (1) of Theorem \ref{repnaffineweyl}.]
We want to show that the lift $R^e_n(s_{v_j})$ is compatible with the
base change map $p_b^*: \V_n\to \V_{ne}$, 
\begin{equation}\label{basechangecompat}
  R^e_n(s_{v_i})(p_b^*(v))=p_b^*(s_{v_i}(v)), \quad i\in \Z/n\Z, ~~v\in
\V_n,
\end{equation}
where $p_b^*$ is defined on the generators as, 
\begin{equation} \label{b*}
p_b^*(v_k)=ew_{ke}+\sum_{i=1}^{e-1}(w_{ke-i}+w_{ke+i}),
\end{equation}
as dictated by Proposition \ref{ssfibrepullback}.
\begin{lemma}\label{S}
\begin{enumerate}
\item For $k\geq 1$,
\begin{equation}\label{S.1}
R^e_n(s_{v_j})(w_{je})= -\sum_{i=(j-1)e+1}^{(j+1)e-1}w_i.
\end{equation}

\item  For $e<|i-je|$,
\begin{equation}\label{S.2}
R^e_n(s_{v_j})(w_{i})=w_{i}.
\end{equation}

\item 
\begin{equation}\label{S.3}
R^e_n(s_{v_j})(w_{(j+1)e})=\sum_{i=0}^ew_{je+i} \quad \mbox{and}\quad
R^e_n(s_{v_j})(w_{(j-1)e})=\sum_{i=0}^ew_{je-i}.
\end{equation}

 \item For $0<i<e$, 
\begin{equation}\label{S.4}
R^e_n(s_{v_j}) (w_{je-i})=w_{(j+1)e-i} \quad \mbox{and}\quad 
R^e_n(s_{v_j})(w_{je+i})=w_{(j-1)e+i}.
\end{equation}
\end{enumerate}
\end{lemma}
\begin{proof} 
The vectors $w(i,j)=w_i+\cdots +w_j$ are orthogonal to all the base
vectors $w_k$, except when $k=i-1, i, j, j+1$. In these cases, 
\[ (w(i,j), w_{i-1})=1, \quad 
(w(i,j), w_{i})=-1, \quad (w(i,j), w_{j})=-1,\quad (w(i,j), w_{j+1})=1.\]
To simplify the indices, we prove the statement taking
  $j=0$. We have, 
 \[ R^e_n(s_{v_0})=s_{w(-e+1,0)}s_{w(-e+2,1)}\cdots s_{w(0,e-1)}.\]
For proving (1), all the reflections except $s_{w(-e+1,0)}$ and
$s_{w(0,e-1)}$ fix the vector $w_0$. Hence, 
\[
\begin{split}
R^e_n(s_{v_0})(w_{0})&=s_{w(-e+1,0)}s_{w(0,e-1)}(w_0)\\
&=s_{w(-e+1,0)}\left(w_0+(w(0,e-1), w_0)w(0,e-1)\right)\\
&=s_{w(-e+1,0)}(w_0-w(0,e-1))=-s_{w(-e+1,0)}(w_1+\cdots w_{e-1})\\
&=-\left(w_2+\cdots w_{e-1}+w_1+ (w_1,w(-e+1,0))w(-e+1,0)\right)\\
&=-\sum_{i=-e+1}^{e-1}w_i.
\end{split}
\]
For the proof of Part (2), we observe that the isometry
$R^e_n(s_{v_0})$ involves only the
Picard-Lefschetz reflections $s_{w_l}$ for $|l|<e$. Each one of
these reflections fixes $w_i$, since  $|l|+1\leq e <|i|$. Hence 
\[R^e_n(s_{v_0}) (w_{i})=w_{i} \quad |i|>e.\]
For the proof of Part (3), reasoning as in the proof of Part (1), 
\[
\begin{split}
R^e_n(s_{v_0})(w_{e})& =s_{w(0,e-1)}(w_e)=w_e+(w_e, w(0,e-1))w(0,e-1)\\
&=\sum_{i=0}^ew_{i}.
\end{split}
\]
The proof of the other equality follows in a similar manner. 

To prove Part (4), we observe that the only reflections occuring in  
$R^e_n(s_{v_0})$ not fixing $w_i$ are the reflections based on the
vectors $w(i-e, i-1)$ and $w(i-e+1, i)$. These vectors are
orthogonal. Thus, 
\[
\begin{split}
R^e_n(s_{v_0})(w_{i})&=s_{w(i-e,i-1)}s_{w(i-e+1, i)}(w_i)
=s_{w(i-e,i-1)}(w_i-w(i-e+1, i))\\
&=w_i+(w(i-e,i-1),w_i)w(i-e,i-1)\\
&~~~~-w(i-e+1, i)-(w(i-e+1, i),w(i-e,i-1))w(i-e, i+1)\\
&=w_i+w(i-e,i-1)-w(i-e+1, i)\\
&=w_{i-e}.
\end{split}
\]
\end{proof}
We now establish Equation \ref{basechangecompat}. To do this, we do it
for $j=0$, and take $v$ to be one of the basis vectors. 
The Picard-Lefschetz isometries  $s_{w_k}$ for 
$0<|k|<e$ involved
in the definition of $R^e_n(s_{v_0})$ correspond to exceptional divisors. As in the proof of Proposition \ref{ssfibrepullback},  the divisors $w_k$ for $0<|k|<e$ are exceptional, and hence do not intersect the pullback divisors $p_b^*(v_i)$. Hence
the reflections $s_{w_k}$ for $0<|k|<e$ fix  $p_b^*(v_i)$. 
The reflection $s_{w_0}$ fixes the pullback vectors $p_b^*(v_i)$
for $|i|>1$. 
Hence  $R^e_n(s_{v_0})$  fixes $p_b^*(v_i)$ when  $|i|>1$and 
the theorem is proved for such basis vectors.  

Hence we are reduced to checking the commutativity for the basis vectors $v_0,
~v_1$ and $v_{-1}$.  Using various parts from Lemma \ref{S}, we obtain
\begin{align*}
R^e_n(s_{v_0}) (p_b^*(v_0))&=R^e_n(s_{v_0})\left( ew_0+\sum_{i=1}^{e-1}
   (e-i)(w_{-i}+w_i)\right) \quad 
\left(\mbox{by Equation \ref{b*}}\right)\\
&=-e\left(\sum_{i=-(e-1)}^{e-1}w_{i}\right)+\sum_{i=1}^{e-1}(e-i)(w_{e-i}+w_{-(e-i)})
\quad \left(\mbox{by Equations \ref{S.1}, \ref{S.3}}\right)\\
&=-\left(ew_0+\sum_{i=1}^{e-1}
   (e-i)(w_{-i}+w_i)\right)=-p_b^*(v_0).
\end{align*}
This proves that 
\begin{equation}\label{commv0}
 R^e_n(s_{v_0})(p_b^*(v_0))=p_b^*(s_{v_0}(v_0))=-p_b^*(v_0).
\end{equation}

We now check the commutativity for the divisor $v_1$ (and the same
proof works  for $v_{-1}$). We can write, 
\[ p_b^*(v_1)=ew_e+\sum_{i=1}^{e-1}
   (e-i)(w_{e-i}+w_{e+i}).\]
By Lemma \ref{S}, 
\begin{align*}
R^e_n(s_{v_0})&\left(ew_e+\sum_{i=1}^{e-1}
   (e-i)(w_{e-i}+w_{e+i})\right)\\
&=eR^e_n(s_{v_0})(w_e)+
R^e_n(s_{v_0}) \left(\sum_{i=1}^{e-1}
   (e-i)w_{e-i}\right)+R^e_n(s_{v_0})\left(\sum_{i=1}^{e-1}
   (e-i)w_{e+i})\right)\\
&= e\sum_{i=0}^ew_{i}+ \sum_{i=1}^{e-1}(e-i)w_{i}+\sum_{i=1}^{e-1}
   (e-i)w_{e+i}\\
&= p_b^*(v_0)+ p_b^*(v_1).
\end{align*}
Hence we get, 
\begin{equation}\label{commv1}
R_n^e(s_{v_0})(p_b^*(v_1))=p_b^*(s_{v_0}(v_1))=p_b^*(v_0+v_1).
\end{equation}
This proves Part (1) of Theorem \ref{repnaffineweyl}.
 \end{proof}

\begin{proof}[Proof of Part (2) of Theorem \ref{repnaffineweyl}.]
We now show that $R_n^e$ defines a representation of $W_n$ to
$W_{ne}$. It follows from Part (2) of Lemma \ref{vij}, that the
transformations $s_{v(i,j)}$ appearing in the definition of
$R^e_n(s_k)$ are reflections that commute with each other. Hence, 
$R^e_n(s_k)^2=1$.

We need to check the braid relations are satisfied by
$R^e_n(s_k)$. For this, it is  convenient to work with the 
permutation realization of these isometries.
Since $n\geq 3$, given any $i, j \in \Z/n\Z$, the braid
relations involving $R^e_n(s_i)$ and $R^e_n(s_j)$ are local. Hence we
can work with the permutation representation of these expressions.  
We write down explicitly, the permutation realization of the 
transformations  $R^e_n(s_k)$ for $k=0, 1, m$: 
\[
\begin{split}
R^e_n(s_0)& =(-e+1,1)(-e+2,2)\cdots (0,e)\\
R^e_n(s_1)& =(1,e+1)(2,e+2)\cdots (e,2e)\\
R^e_n(s_m)& =((m-1)e+1,me+1)((m-1)e+2,me+2)\cdots (me,(m+1)e), 
\end{split}
\]
where we have used the equality sign to denote the realization 
as permutations on the set $\Z/ne\Z$.
The transpositions $(k, k+e)$ for $k\leq 0$ appearing in the realization of 
$R^e_n(s_0)$ and the transposition $(l, l+e)$ for $(m-1)e+1\leq l\leq me$ 
appearing in the realization  of 
$R^e_n(s_m)$ for $|m|\geq 2$ commute with 
each other. Hence it follows that $ R^e_n(s_0)$ and $R^e_n(s_m)$ for $|m|\geq 2$
commute. 

It remains to show that $(R^e_n(s_0)R^e_n(s_1))^3=1$. The transposition
$(k, k+e)$ for $k\leq 0$
commutes with the transpositions  $v(l, l+e)$ for $1\leq l\leq e$ 
except when $l=k+e$. The product $(k, k+e)(k+e, k+2e)$ is order $3$.  
Hence it follows
$(R^e_n(s_0)R^e_n(s_1))^3=1$.

A similar calculation applies by replacing the indices $0, 1 $ and
$m$, and this proves Part (2) of Theorem \ref{repnaffineweyl}. 
\end{proof}

\begin{proof}[Proof of Part (3) of Theorem \ref{repnaffineweyl}.]
We now want to prove that the family of representations of the affine
Weyl groups we constructed satisfy the composition relation: 
\[ R^f_{ne}\circ R_n^e=R_n^{ef}, \]
where $n, e, f$ are any natural numbers. This statement is the 
compatibility relation with respect to the composition of pullbacks
that is required of the universal Picard-Lefschetz isometries. 
We first compute the lifts $ R_n^e(s_{v(i,j)})$ 
of the reflection  $s_{v(i,j)}$ based at the vector $v(i,j)=v_i+\cdots
+v_j$: 
\begin{lemma} \label{lem:svijperm}
For $k\geq 1$, the permutation realization of  $R^e_n(s_{v(j,j+k-1)})$ is given by
\[ ((j-1)e+1, (j+k)e+1)\cdots (je,(j+k)e).  \]
\end{lemma}
\begin{proof} The proof is by induction on $k$. We take $j=0$ and for
  $k=1$, the permutation realization of $R^e_n(s_{v_0})$ is
$(-e+1,1)(-e+2,2)\cdots (0,e)$. 
Assume that the lemma has been proved for $k-1$. By Part (4) of Lemma
\ref{vij}, $s_{v(0,k)}=s_{v_k}s_{v(0,k-1)}s_{v_k}$. Hence the
permutation  realization of $R^e_n(s_{v(0,k)})$ is given by,  
\begin{align*}
  &\{((k-2)e+1,(k-1)e+1)\cdots ((k-1)e,ke)\}\\
  & \{(-e+1, (k-1)e+1)\cdots (0,(k-1)e)\}\\
& \{((k-2)e+1,(k-1)e+1)\cdots ((k-1)e,ke)\}\\
=&(-e+1, ke+1)\cdots (0,ke),
\end{align*}
and this proves the lemma. 
\end{proof}
From the defintion of $R^e_n(s_{v_0})$, we get 
\[ R_{ne}^f(R_n^e(s_{v_0}))=R_{ne}^f(s_{w(-e+1,0)})\cdots R_{ne}^f(s_{w(0,e-1)}).\]
Upon substituting  $n=ne$ and $e=f$, in the equation given by 
Lemma \ref{lem:svijperm}, the permutation realization 
(as permutations on $\Z/nef\Z$) of  $R_{ne}^f(s_{w(j,j+e)})$ is, 
\[ ((j-1)f+1, (j+e)f+1)\cdots (je,(j+f)e).\]
Hence the permutation realization of $R_{ne}^f(R_n^e(s_{v_0}))$ is given by,  
\[ (-fe+1,1)(-fe+2,2)\cdots (0,fe),\]
which is equal to the permutation realization of 
$R_n^{ef}(s_{v_0})$.

As all these expressions are local, the equality as 
permutations establishes Part (3) of  Theorem \ref{repnaffineweyl} for the
reflection $s_{v_0}$. By symmetry it establishes for the other
generators. Since we know that the collection of maps $R_n^e: W_n\to
W_{ne}$ define homomorphisms  as $n$ and $e$ varies, this establishes
Part (3) of  Theorem \ref{repnaffineweyl}. 
\end{proof}
\begin{corollary}\label{cor:univisom}
Let $n\geq 3$. For any $x\in W_n$, the collection of elements
$R_n^e(x)\in W_{ne}$ are compatible isometries in the following sense:
for any natural numbers $e, ~f$ and $w_j\in \V_{ne}$, 
\[ R_{ne}^f(x)(p_b^*(w_j))=p_b^*( R_{n}^e(x)(w_j)), \]
where \[p_b^*(w_j)=fz_{jf}+\sum_{l=1}^{f-1}(f-l)(z_{fj-l}+z_{fj+l}),\]
is the base change map defined from $\V_{ne}\to \V_{nef}$ as in
Proposition \ref{ssfibrepullback}, with standard bases
$w_j, j\in \Z/ne\Z$ and $z_l, l\in \Z/nef\Z$ for $\V_{ne}$ and $\V_{nef}$
respectively. 
\end{corollary} 

\section{Universal isometries: Proof of Theorem \ref{upl}}
\label{sec:upl} We  now show that the Picard-Lefschetz reflections
define universal isometries of the family of N\'eron-Severi lattices
$NS(X_b)$ as $b$ varies.  Suppose $X\to C$ a semistable,  elliptic
surface and the Kodaira fibre type at a point $x_0\in C(k)$ is of type
$I_n$ with $n\geq 3$. Given an irreducible component $v$ of the
singular fibre at $x_0$, the map 
\[ s_v(x)=x+ <x,v>v, \quad x\in NS(X),\] 
defines the Picard-Lefschetz
reflection based at $v$ of $NS(X)$. Let $v_0, \cdots, v_{n-1}$ be the
irreducible components of the singular fibre $p^{-1}(x_0)$. The
reflections $s_{v_i}$ generates the affine Weyl group $W_n(x_0)$ based
on the fibre $x_0$, giving  an action of $W_n(x_0)$ on the
N\'eron-Severi group $NS(X)$ of $X$.

Suppose $b: B\to C$ is a finite, separable map in $\cB_C$,
and let $y_1, \cdots, y_r$ be the points of $B$ lying above $x_0$.
We use the variable $y$ to denote one of the fibres. 
Suppose that the local ramification degree at $y$ is $e_y$. Let 
$w_0^y, \cdots, w_{ne_y-1}^y$ be the irreducible components of the
singular fibre 
$p_b^{-1}(y)$. By the results of Section \ref{sec:repn}, there is a
representation $R_y :W_n(x_0)\to {\rm Aut}(NS(X))$ defined on 
the Picard-Lefschetz reflection based on the 
irreducible component $v_k$ of the fibre at $x_0$ as, 
\[ R_y(s_{v_k})=\prod_{w(i,j)\in I(n,e_y,k)}s_{w^y(i,j)},\] 
where the set  $I(n,e_y, k)$ is the collection of vectors of the form
$w^y(i,j)=\sum_{l=i}^jw_l^y$ of length $e_y$ and support containing
$ke_y$.  Define 
\[PL_b(s_{v_k})=\prod_{y\in b^{-1}(x_0)}R_y(s_{v_k}).\]
By construction, $\theta_b(s_{v_k})$ is an element of ${\rm
  Aut}(NS(X))$. 

In order to prove Theorem \ref{upl},  that  $PL_b(s_{v_k})$ defines a
universal isometry, it needs to be checked its  compatibility with the
base change map $p_a^*: NS(X_b) \to NS(X_{b\circ a})$ for maps
$A\xrightarrow{a} B\xrightarrow{b} C$.  On the fibral divisors this
compatibility is given by Corollary \ref{cor:univisom}. We need to
check it only on sections. 

Suppose $(P)$ is a section of $\pi$ not passing through $v_0$. Then
$s_{v_0}$ fixes $(P)$. The pullback section $p_b^*((P))$ intersects
the fibre over $y$ at one of the components $w_{ke_y}, ~k\neq 0$.
Since the definition of $PL_b(s_{v_0})$ involves onlythe reflections
corresponding to exceptional divisors $w_i, 0\leq |i|<e_y$
$PL_b(s_{v_0})$ fixes $p_b^*((P))$. 

Now lets assume that  $(P)$ be a section of $\pi$ passing through
$v_0$. Let $w_0^{y_i}$ be the identity component at the fibre over
$y_i$ of the pullback divisor $p_b^*(v_0)$. The pullback section
$p_b^*(P)$,  is a section of $X_b\to B$ passing through $w_0^{y_i}$
for $i=1,\cdots, r$.  Using the fact that the reflections appearing in
the definition of $R_y(s_{v_0}) $ are mutually orthogonal we get, 
\[ 
\begin{split} 
 R_y(s_{v_0})(p_b^*(P))&=s_{w^y(-e_y+1,0)}s_{w^y(-e_y+2,1)}\cdots s_{w^y(0,
   e_y-1)}(p_b^*(P))\\
&= p_b^*(P)+{w^y(-e_y+1,0)}+ {w^y(-e_y+2,1)}+\cdots +{w^y(0,e_y-1)}\\
&=p_b^*(P)+e_yw^y_0+\sum_{j=1}^{e-1} 
(e-j)(w^{y}_{-j}+w^{y}_j). 
\end{split}
\]

 Then, 
\begin{align*}
 PL_b(s_{v_0})((P))&=R_{y_1}(s_{v_0})\cdots R_{y_r}(s_{v_0}) ((P))\\
&=p_b^*(P) +\sum_{i=1}^r\sum_{j=1}^{e_i}\left(e_iw_0^{y_i}+\sum_{j=1}^{e_i-1} 
(k-j)(w_{-j}^{y_i}+w_j^{y_i})\right)\\
&= p_b^*((P)+v_0)=p_b^*(s_{v_0}((P)).
\end{align*}
This proves the compatibility of $PL_b(s_{v_0})$ with the pullback
map on sections, thereby showing that it defines a universal
isometry, and finishes the proof of Theorem  \ref{upl}. 

\section{Proof of Theorem \ref{guti}}\label{sec:guti} Let $S$ be the
singular locus of $\pi: X\to C$. For $t\in S$, let the singular fibre
be of Kodaira type $I_{n_t}$.  The space $N(X_t)$, the subspace of
$N(X)$ generated by the components of the fibre of $\pi$  based at
$t$,  equipped with its intersection pairing is isomorphic to the root
lattice of type $\tilde{A}_{n-1}$.  Let $A(N(X_t))$ be the
automorphism group of $N(X)$ generated by the Picard-Lefscetz
transformations based on the irreducible components of the fibre
$X_t$. The group  $A(N(X_t))$ is isomorphic to the affine Weyl group
$W_{n_t}$.

It follows from Theorems \ref{repnaffineweyl} and \ref{upl}, that 
the maps $PL_b$ can be extended multiplicatively to give a representation 
of the product of the affine Weyl groups over $t\in S$ as universal
isometries of the elliptic surface $\cE$:
 \[ PL: \prod_{t\in S} A(N(X_t))\to {\rm Aut}(UNS(X)).\] Let $\Phi\in
{\rm Aut}(UNS(X))$ be an isometry of $UNS(X)$.  By Proposition
\ref{fibre}, after multiplying by the automorphism $-1$ if required,
we can assume that $\Phi(F)=F$.  By Proposition \ref{fibrepreserve},
$\phi:NS(X)\to NS(X)$ restricts to an isometry $N(X_t)\to N(X_t)$ for
$t\in S$. 

The space $N(X_t)$ can be identified with the root lattice of the
affine root system $\tilde{A}_{n-1}$. Let $v_i^t, ~i\in \Z/n_t\Z$ be a
standard basis for $N(X_t)$. We have two bases for this affine root
system: $\{v^t_0, \cdots, v^t_{n_t-1}\}$ and  $\{\phi(v^t_0), \cdots,
\phi(v^t_{n_t-1})\}$.  

By \cite[Proposition 5.9]{Kac}, there exists an element $x_t\in
A(N(X_t))\simeq W_{n_t}$, that maps the basis  $\{\phi(v^t_0), \cdots,
\phi(v^t_{n_t-1})\}$ to the standard basis $\{v^t_0, \cdots,
v^t_{n_t-1}\}$ or to its negative. Since any element of $A(N(X_t))$ is
generated by the Picard-Lefschetz transformations, which preserve the
fibre $F=v^t_0+\cdots +v^t_{n_t-1}$, so does $\phi$. It follows that
$x_t$ takes the basis $\{\phi(v^t_0), \cdots, \phi(v^t_{n_t-1})\}$ to
the standard basis $\{v^t_0,\cdots, v^t_{n_t-1}\}$. 

By Theorems \ref{upl} and \ref{repnaffineweyl}, we can assume that
$x_t$ defines (universal) automorphisms of $UNS(X)$. Define $\Psi\in
{\rm Aut}(UNS(X)$ by
\[ \Psi=\Phi\circ \prod_{t\in S}PL(x_{t}).\] Denote by $\psi$ its
restriction to $NS(X)$. We have, 

\begin{proposition}\label{prop:nesec}
With notation as above, $\psi$ maps sections to sections. 
\end{proposition}
\begin{proof}
The property  of $\psi$ that we require in the proof is that 
$\psi$ preserves the standard basis for each singular fibre of $X\to C$.
In particular, this implies that   
$\psi(F)=F$.
By renaming if required, it is enough to show that the
zero section   $(O)$ is mapped to a section by $\psi$. Write,   
\[ \psi((O))= (P)+V+rF,\]
where $V$ is a fibral divisor.  It is enough to show
that after translation by $(-P)$, $\psi((O))$ is a section. Hence we
can assume that $(P)=(O)$.  We need to show that $V$ and $r$ are
zero. Write $V=\sum_{t\in S}V_t$, where $V_t$ is the contribution 
to $V$ from $N(X_t)$.
We argue fibrewise and first show that each  $V_t$ is zero, upto modifying $r$. 
 
Fix $t$ and for notational ease, we drop the superscript $t$.   Suppose
that $\psi(v_0)=v_k$ for some $k\neq 0$, and $\psi(v_j)=v_0$. 
Write $V=\sum_{i\in \Z/n\Z}a_iv_i$. Modify $r$, such that $a_0=0$.$k=j$.
For $l\neq 0, ~k$, the equation 
\[0=(O).\psi^{-1}(v_l)=\psi((O)).v_l=(O).v_l+V.v_l=-2a_l+a_{l+1}+a_{l-1},\]
yields the equality $a_{l+1}=2a_l-a_{l-1}$.
Going from $0$ to $k$ in the increasing order, we get $a_l=la_1$ for $l\leq k$.
Going from $0=n$ to $k$ in the reverse order, 
we get $a_{-l}=la_{-1}$ for $l\leq n-k$. 
Hence we get $ka_1=a_k=a_{-(n-k)}=(n-k)a_{-1}$.
From the equation,
 \[0=(O).v_j=\psi((O)).\psi(v_j)=(O).v_0+V.v_0=1+a_{1}+a_{-1},\]
 we get $a_1=-(1+a_{-1})$. Combining these two equations gives,
 \[-k(1+a_{-1})= (n-k)a_{-1}, \quad \mbox{i.e.,} \quad  -k=na_{-1}.\]
 Since $0<k<n$, this implies
$a_{-1}$ is non-integral, contradicting the integrality 
of the coefficients $a_j$ of $V$. 

Hence this implies that $k=0$, i.e., $\psi(v_0)=v_0$, and hence
$\psi(v_i)=v_i$ or $v_{-i}$.  In either case, for $i\neq 0$, 
\[0=(O).v_i=\psi((O)).\psi(v_i)=(O).\psi(v_i)+V.\psi(v_i)=V.\psi(v_i).\]
As the space generated by the vectors $v_i$ for  $i\neq 0$ is negative 
definite, this implies $V=0$ and  $\psi((O))= (O)+rF$, for some 
integer $r$. Considering self-intersections, 
\[ -\chi(X)=(O)^2=\psi((O))^2=(O)^2+2r=-\chi(X)+2r,\]
we get that $r=0$ and hence $\psi((O))$ is a section. This proves the
proposition. 
\end{proof}

As a consequence of this proposition,  translating by a section if required,
we can assume that $\psi((O))=(O)$. 

Since $\psi((O))=(O)$, it follows that $\psi(v_0^t)=v_0^t$ for $t\in
S$ (this was proved as part of the proof of the proposition). Hence
for any $t$ and $1\leq i<n_t$,  $\psi(v^t_i)=v^t_i$ or
$v^t_{-i}$. In particular, $\psi$ restricts to an involution
restricted to $N(X_t)$ for each $t\in S$. We would like to extend
these properties to the universal isometry $\Psi$: 

\begin{proposition}\label{psitoPsi}
Let $\cE: X\to C$ be a semistable elliptic surface. Let $S$ be the singular
locus, and assume that the singular fibre at $t\in S$ is of
Kodaira-N\'eron type $I_{n_t}$ with $n_t\geq 3$. Suppose $\Psi$ is an
universal isometry of $UNS(\cE)$ such that  $\psi=\Psi|_{NS(X)}$
satisfies the following property (E): 

(E): For all $t\in S$, $\psi(v^t_i)=v^t_i$ or
$v^t_{-i}$, where $\psi=\Psi|_{NS(X)}$ and $\{v_i^t, ~i\in \Z/n_t\Z\}$
are the irreducible components of the singular fibre at
$t$, and $v_0^t$ is the component meeting the section $(O)$. 

Then for every $b\in \cB_C$, the map $\psi_b$ satisfies Property
$E$. Further,  $\Psi$ is uniquely determined by $\psi$. 
\end{proposition}  
\begin{proof}
Fix a point $t_0\in S$ and a point $t^b_0$ of $B$ lying above
$t_0$. Suppose that the Kodaira type of the fibre over $t_0$
(resp. $t^b_0$) is $I_n$ (resp. $I_{ne}$). We can assume $e>1$. 
Denote the irreducible
components of the fibre $X_{t_0}$ by $v_i, ~i \in \Z/n\Z$
and those over $t^b_0$ by $w_i, ~i\in \Z/ne\Z$, where $v_0, ~w_0$ are
the components meeting the zero section. 

For  $k\in  \Z/ne\Z$,  $\psi_b(w_k)^2=w_k^2=-2$. By Part (3) of Lemma
\ref{vij},  $\psi_b(w_k)=w(i_k,j_k)+r_kF$ for some integers $i_k,
~j_k, ~r_k$. Suppose $w(i, j)$ and $w(k,l)$ are two vectors whose
supports intersect. We have, 
\[w(i, j),w(k,l)=\begin{cases} -2 &\text{if $i=k$ and $j=l$}.\\
-1 &\text{if either $i=k$ or $j=l$ and supports not equal},\\
0 &\text{if none of the endpoints are equal}.
\end{cases}
\] Since  $\psi_b(w_k)\psi_b(w_{k+1})=1$, it follows that the supports
of     $w(i_k, j_k)$ and $w(i_{k+1}, j_{k+1})$ do not intersect, and
the union $[i_k, j_k]\cup [i_{k+1}, j_{k+1}]$ forms a connected
segment. If for some $k$, the segment $[i_{k+2}, j_{k+2}]$ intersects
the segment $[i_{k}, j_{k}]$, then at least one of their endpoints
have to coincide. By the above calculation,
$\psi_b(w_k)\psi_b(w_{k+2})$ is either $-2$ or $-1$, contradicting the
fact that it is equal to $w_kw_{k+2}=0$ (as $e>1$, $ne\geq 5$). Hence
the disjoint segments  $[i_{k}, j_{k}]$ as $k$ varies join together to
form a connected segment without any back tracking, and fill up
$\Z/ne\Z$. These conditions force for each $k$, $i_k=j_k$. Now,
\[\psi_b(w_0).\psi_b(p_b^*(v_0))=w_0.p_b^*(v_0)=-2e+2(e-1)=-2.\]  
The support of the pullback divisor $p_b^*(v_0)$ is the set $|j|<e$. The
exceptional divisors $w_j, ~0<|j|<e$ do not intersect
$p_b^*(v_0)$. These conditions force $\psi_b(w_0)=w_0+r_0F$.  It
follows that $\psi_b(w_k)=w_{\pm k}+r_kF$ for $k\in \Z/ne\Z$.
Intersecting with the zero section, we get $r_k=0$ for all $k$. 

Suppose $\psi(v_k)=v_{-k}$ for all $k\in \Z/n\Z$. Since $\Psi$  is a
universal isometry, 
\[p_b^*(v_k)=p_b^*(\psi(v_{-k}))=\psi_b(p_b^*(v_{-k})).\]  
Hence we have,
\[e\psi_b(w_{-ke})+\sum_{i=1}^{e-1}(e-i)(\psi_b(w_{-ke+i})+\psi_b(w_{-ke-i}))
=ew_{ke}+\sum_{i=1}^{e-1}(e-i)(w_{ke+i}+w_{ke-i}).\] 
This forces $\psi_b(w_{ke})=w_{-ke}$ for $k\in \Z/n\Z$.  The hypothesis $n\geq 3$,
together with the fact proved above   forces  $\psi_b(w_{k})=w_{-k}$
for $k\in \Z/ne\Z$.  A similar argument works if we had assumed that
$\psi(v_k)=v_{k}$ for all $k\in \Z/n\Z$, forcing in this case $\psi_b$
to be identity on the fibres above $t_0$.

It is clear that not only have we proved that $\Psi$ is uniquely
determined by $\psi$, but in fact that the behaviour of $\psi_b$ on a
singular fibre  at $s\in B(k)$ is similar to that of $\psi$ on
$b(s)\in C(k)$, in whether it acts as the identity or flips around the
origin according respectively to the behaviour of $\psi$.  
\end{proof}

\begin{proof}[Proof of Theorem \ref{guti}] We are now in a position to
describe the automorphism group of the universal N\'eron-Severi group.
Given an universal isometry $\Phi$, by Proposition \ref{fibre}, we
first multiply  by $-1$ if required to ensure that $\Phi$ fixes the
fibre. By Proposition \ref{singlocus}, the resulting automorphism
restricts to an automorphism of $N(X_t)$ for each point $t\in S$, the
ramification locus of $\pi$. By the argument given before the
statement of Proposition \ref{prop:nesec}, modify $\Phi$ by an element
of the form $PL(x_{t})$ for some element  $x_t\in A(N(X_t))$ to ensure
that the base morphism $\phi$ maps the standard basis of any singular
fibre of $\pi$ to the standard basis.

This ensures, by  Proposition \ref{prop:nesec}, that $\phi$ preserves
sections of $\pi$. Now we modify $\phi$ by a translation to ensure
that the zero section $(O)$ of $\pi$ is fixed. By Proposition
\ref{psitoPsi}, each $\phi_b$ for $b\in \cB_C$ preserves the standard
basis of each fibre. In particular $\Phi$ preserves the irreducible
components of the singular fibres. By  Proposition \ref{prop:nesec},
applied to each $b\in \cB_C$, $\Phi$ maps sections to sections. 

This ensures that the hypothesis of Theorem  \ref{torelli} hold. As a
consequence, $\Phi$ is either induced by the inverse map on the
generic fibre or is the identity map. Hence the automorphism group is
generated by the above transformations.

The transformation sending $x\mapsto -x, ~x\in UNS(X)$ is central in
the automorphism group. Given an universal automorphism $\Phi$ of
$UNS(X)$, it fixes the trivial lattice, and hence gives a compatible
family of automorphisms, as $b$ varies in $\cB_C$, of the Mordell-Weil
groups of the generic fibre $E(l(B))$. Since the Picard-Lefscetz
tranformations act trivially on the Mordell-Weil groups, the kernel of
this homomorphism is the group $PL\left(\prod_{t\in S}
A(N(X_t))\right)$. The group generated by the translations by
sections, the automorphism of the generic fibre and the central $-1$
element, project isomorphically as automorphisms of the Mordell-Weil
lattices. This proves the semi-direct property of the automorphism
group. 
\end{proof}

\begin{acknowledgement}
 We thank R. V.  Gurjar, D. S. Nagaraj and M. Rapoport for useful discussions. We thank J.-L. Colliot-Th\'{e}l\`{e}ne, P. Colmez and 
 B. Kahn for stimulating discussions at the Indo-French conference held at Chennai, 2016. 
The first named author thanks McGill University, Montreal and
Universit\'e de Montreal for
productive stays during the periods 1994-96 and for three months in
1998 when some of these questions first arose. He thanks H. Kisilevsky for
the reference to Zarhin's question (on a trip from Montreal to
Vermont, when we realized we are working on similar questions). The
first named author thanks MPIM, Bonn for providing an excellent hospitality and
working environment during his stay there in May-June 2016, when 
working on this problem. 
\end{acknowledgement}

\end{document}